\documentclass[sn-mathphys-num]{sn-jnl}% Math and Physical Sciences Numbered Reference Style 
\usepackage{graphicx}%
\usepackage{multirow}%
\usepackage{amsmath,amssymb,amsfonts}%
\usepackage{amsthm}%
\usepackage{mathrsfs}%
\usepackage[title]{appendix}%
\usepackage{xcolor}%
\usepackage{textcomp}%
\usepackage{manyfoot}%
\usepackage{booktabs}%
\usepackage{algorithm}%
\usepackage{algorithmicx}%
\usepackage{algpseudocode}%
\usepackage{listings}%

\usepackage{tikz}
\usetikzlibrary{positioning}
\usetikzlibrary{calc,through,backgrounds}
\usepackage{pgfplots}
\pgfplotsset{compat=newest}
\usepackage{changepage}
\theoremstyle{thmstyleone}%
\newtheorem{proposition}{Proposition}[section]% 

\newtheorem{remarque}{Remark}[section]
\newtheorem{lemme}{Lemma}[section]
\newtheorem{corollaire}{Corollary}[section]

\theoremstyle{thmstyletwo}%

\theoremstyle{thmstylethree}%
\newtheorem{definition}{Definition}%

\newcommand{\1}{{\mathchoice {\rm 1\mskip-4mu l} {\rm 1\mskip-4mu l}{\rm 1\mskip-4.5mu l} {\rm 1\mskip-5mu l}}}

\newcommand{\N}{\mathbb{N}}
\newcommand{\R}{\mathbb{R}}
\newcommand{\prox}{\text{prox}}
\let\xp\undefined
\newcommand{\xp}{\xi_p} 
\newcommand{\bp}{\zeta_p} 
\newcommand{\zp}{\omega_p} 
\newcommand{\X}{\mathscr{X}} % Estimation space
\newcommand{\Y}{\mathscr{Y}} % Observation space
\newcommand{\Z}{\mathscr{Z}} % Observation space
\begin{document}

\title[Stability Bounds for the Unfolded Forward-Backward Algorithm]{Stability Bounds for the Unfolded Forward-Backward Algorithm}


\author[1]{\fnm{Emilie} \sur{Chouzenoux}}\email{emilie.chouzenoux@inria.fr}

\author[2]{\fnm{C\'ecile} \sur{Della Valle}}\email{cecile.della-valle@parisdescartes.fr}

\author*[1]{\fnm{Jean-Christophe} \sur{Pesquet}}\email{jean-christophe.pesquet@centralesupelec.fr}

%Université de Paris, Paris Sorbonne Université

\affil[1]{\orgdiv{Center for Visual Computing}, \orgname{Inria, CentraleSup\'elec, University Paris Saclay}, \orgaddress{\city{Gif-sur-Yvette}, \postcode{91190}, %\state{State},
\country{France}}}

\affil[2]{\orgdiv{Universit\'e de Paris}, \orgname{Paris Sorbonne Universit\'e}, \orgaddress{
%\street{Street}, 
\city{Paris}, \postcode{75005}, %\state{State}, 
\country{France}}}

% \affil[3]{\orgdiv{Department}, \orgname{Organization}, \orgaddress{\street{Street}, \city{City}, \postcode{610101}, \state{State}, \country{Country}}}


\abstract{We consider a neural network architecture designed to solve inverse problems where the degradation operator is linear and known. This architecture is constructed by unrolling a forward-backward algorithm derived from the minimization of an objective function that combines a data-fidelity term, a Tikhonov-type regularization term, and a potentially nonsmooth convex penalty. The robustness of this inversion method to input perturbations is analyzed theoretically. Ensuring robustness complies with the principles of inverse problem theory, as it ensures both the continuity of the inversion method and the resilience to small noise - a critical property given the known vulnerability of deep neural networks to adversarial perturbations.
A key novelty of our work lies in examining the robustness of the proposed network to perturbations in its bias, which represents the observed data in the inverse problem. Additionally, we provide numerical illustrations of the analytical Lipschitz bounds derived in our analysis.}




\maketitle

\section{Introduction}\label{sec1}

\paragraph{Inverse problems.}
A large variety of inverse problems consists of inverting convolution operators such as signal/image restoration~\cite{Bertero2009,Chouzenoux2012}, tomography~\cite{Wu2020}, Fredholm equation of the first kind~\cite{Arsenault2017}, or inverse Laplace transform~\cite{Cherni2017}.

Let $T$ be a bounded linear operator from a separable Hilbert space $\X$ to a Hilbert space $\Y$.
The problem consists in estimating  $\overline{x} \in \X$ from observed data
 \begin{equation} \label{def:InversePb}
 y = T\overline{x} + w
 \end{equation}
 where $w$ corresponds to an additive measurement noise.
The above problem is often ill-posed i.e., 
 a solution might not exist, 
 might not be unique, 
 or might not depend continuously on the data.

 \paragraph{Variational problem.} 
The well-posedness of the inverse problem defined by~\eqref{def:InversePb} is retrieved by regularization.
Here we consider Tikhonov-type regularization.
Let $\tau \in ]0,+\infty[$ be the regularization parameter.
Solving the inverse problem~\eqref{def:InversePb} with such regularization, leads to the resolution of the following optimization problem:
\begin{equation}
\label{def:varJ0}
\underset{x \in C}{\text{minimize}}\, 
J_{\tau}(x)
\;,
\end{equation}
where 
\begin{equation}\label{e:defJtau}
(\forall x \in \X)\quad  J_\tau(x) = \frac{1}{2} \| Tx - y \|^2 + \frac{\tau}{2} \| D x \|^2,
\end{equation}
$C$ is a nonempty closed convex subset  of $\X$ encoding some prior knowledge, e.g. some range constraint or sparsity pattern,
while $D$ acts as a derivative operator of some order $r \geq 0 $.
Often, we have an a priori of smoothness on the solution 
which justifies the use of such a derivative-based regularization.\footnote{$\|\cdot\|$ denotes the norm of the considered Hilbert space.}
Problem \eqref{def:varJ0} is actually an instance of the more general class of problems stated below, encountered in many signal/image processing tasks:
\begin{equation}
\label{def:varJ}
\underset{x \in \X}{\text{minimize}}\,
J_{\tau}(x)+\mu\, g(x)
\;,
\end{equation}
where $\mu \in [0,+\infty[$ is an additional regularization constant and
$g$ is a proper lower-semicontinuous convex function from Hilbert space 
$\X$ to $]-\infty,+\infty]$. Indeed, Problem \eqref{def:varJ0} corresponds to the case when $g$ is the indicator function $\iota_C$
of set $C$.

\paragraph{Neural network.}
We focus our attention on seeking for a solution to the addressed inverse problem through nonlinear approximation techniques
making use of neural networks. 
Thus, instead of considering the solution to the regularized problem~\eqref{def:varJ},
we define the solution to the inverse problem~\eqref{def:InversePb}
as the output of a neural network, whose structure is similar to  a recurrent network~\cite{Yu2019,Tang2022}. 

Namely, by setting an initial value $x_0$, we are interested in the following $m$-layers neural network
where $m\in \mathbb{N}\setminus\{0\}$:
\begin{equation}
	\label{def:oldmodelNN}
	\begin{cases}
	\textbf{Initialization:} \\
	\quad b_0 = T^* y ,\\
	\textbf{Layer $n\in \{1,\ldots,m\}$:} \\
      \quad x_n 
          = \prox_{\lambda_n \mu_n g}\big(x_{n-1}- \lambda_n (T^*T+D_n^*D_n)x_{n-1} + \lambda_n b_0\big)\;.
    \end{cases}		  
\end{equation}
%where, for every $n\in \{1,\ldots,m\}$,
%\begin{align}
%&R_n = \prox_{\lambda_n \mu_n g}
%\label{e:defRn}\\
%&W_n =  \1 - \lambda_n (T^*T+D_n^*D_n) 
%\label{e:defWn}\\
%&V_n = \lambda_n \1.
%\end{align}
Hereabove, $\prox_{\varphi}$ stands for the proximity operator of a lower-semicontinuous proper convex function $\varphi$ (see Section~\ref{section:notation}) and, for every $n\in \{1,\ldots,m\}$, $D_n$ is a bounded linear operators from $\X$ to some Hilbert space
$\Z$, and  $\lambda_{n}$ is positive real.
Throughout this paper, $L^*$ denotes the adjoint of a bounded linear operator~$L$ defined
on suitable Hilbert spaces.
The overall structure of the network is shown in Figure~\ref{fig:structure}.

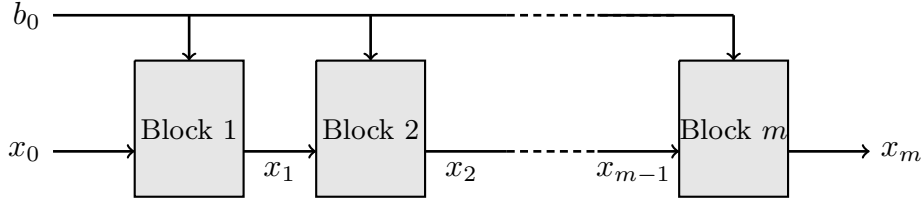
\begin{figure}[h]
\centering
\begin{tikzpicture}[scale=1.2, every node/.style={scale=1.2}]
% Style
\tikzstyle{block}=[thick, fill=gray!20]
% Initialisation
\draw [line width=1pt] (0,1.5) node[left]{$b_0$} ; 
\draw [line width=1pt] (0,0) node[left]{$x_0$} ; 
% Fleche du biais
\draw [line width=1pt] (0,1.5) -- (5,1.5) ;
% ========================================
% BLOCK #1
\draw [line width=1pt][->](0,0) -- (0.9,0) ; % x input block
\draw [line width=1pt][->](1.5,1.5) -- (1.5,1) ;   % biais
\draw [block] (0.9,1) -- (2.1,1) -- (2.1,-0.5) -- (0.9,-0.5) -- (0.9,1) ;
\node[font=\scriptsize] (B1) at (1.5,0.25) {Block $1$};
\draw [line width=1pt][->](2.1,0) -- (2.9,0) ; % x output
\draw (2.5,0) node[below]{$x_1$} ; 
% ========================================
% BLOCK #2
\draw [line width=1pt][->](3.5,1.5) -- (3.5,1) ;   % biais
% Block
\draw [block] (2.9,1) -- (4.1,1) -- (4.1,-0.5) -- (2.9,-0.5) -- (2.9,1) ;
\node[font=\scriptsize] (B2)at(3.5,0.25) {Block $2$};
\draw [line width=1pt](4.1,0) -- (5,0) ; % x output
\draw (4.5,0) node[below]{$x_2$} ; 
% ========================================
% BLOCK (...)
\draw [line width=1pt,densely dashed](5,0) -- (6,0) ; %x (...)
\draw [line width=1pt][->](6,0) -- (6.9,0) ; % x fleche
\draw (6.4,0) node[below]{$x_{m-1}$} ; 
\draw [line width=1pt, densely dashed](5,1.5) -- (6.9,1.5) ; % biais (...)
% Block
\draw [block] (6.9,1) -- (8.1,1) -- (8.1,-0.5) -- (6.9,-0.5) -- (6.9,1) ;
\node[font=\scriptsize] (B3)at(7.5,0.25) {Block $m$};
\draw [line width=1pt] (6,1.5) -- (7.5,1.5) ; % bias till end
\draw [line width=1pt][->](7.5,1.5) -- (7.5,1) ;% biais fleche
% ========================================
% END BLOCK 
\draw [line width=1pt][->](8.1,0) -- (9,0) ; % x fleche
\draw (9,0) node[right]{$x_{m}$} ; 
\end{tikzpicture}
\caption{Global architecture of the $m$-layers neural network~\eqref{def:oldmodelNN}.}
\label{fig:structure}
\end{figure}

A simple choice consists in setting, for every $n\in \{1,\ldots,m\}$, $D_n = \sqrt{\tau_n} D$ where $\tau_n$ is a 
positive constant. The constants $(\mu_n)_{1\le n \le m}$,
$(\tau_n)_{1\le n\le m}$, and possibly $(\lambda_n)_{1\le n \le m}$ can be learned during training. 
Then, Model~\eqref{def:oldmodelNN} can be viewed as unrolling $m$ iterations of an optimization algorithm, so leading to Algorithm~\ref{algo:NN}. If additionally, we set $\mu_n \equiv \mu$ and $\tau_n \equiv \tau$,
Algorithm~\ref{algo:NN} identifies with the forward-backward algorithm~\cite{combettes2005signal,combettes2011proximal} \textcolor{black}{(also called proximal gradient algorithm)}
 applied to the variational problem~\eqref{def:varJ}.

  \begin{algorithm}
 	\caption{Proximal forward-backward splitting method} 
  \label{algo:NN}
 	\begin{algorithmic}[1]
		\State Set $x_0$ ,
 		\For {$n=1,2,\ldots,m$} 
 		    \State Set $ \lambda_n, \tau_n, \mu_n $ ,
 			\State $x_{n} = \text{prox}_{\lambda_n \mu_n g} \; \left( x_{n-1} - \lambda_n \nabla J_{\tau_n} (x_{n-1}) \right) $ ,
 		\EndFor
 		\State \textbf{return} $x_{m}$
 	\end{algorithmic} 
 \end{algorithm}

 Let $\overline{\chi} \ge 0$ denote the modulus of strong convexity of the regularization term $g$ (with $\overline{\chi} > 0$ if $g$ is strongly convex). We have thus
$g = g_0+\overline{\chi}\|\cdot\|^2/2$ where 
$g_0$ is a proper lower-semicontinuous convex function from $\X$ to
$]-\infty,+\infty]$, and  Model~\eqref{def:oldmodelNN} can be equivalently expressed as
\begin{equation}
	\label{def:modelNN}
	\begin{cases}
	\textbf{Initialization:} \\
	\quad b_0 = T^* y ,\\
	\textbf{Layer $n\in \{1,\ldots,m\}$:} \\
      \quad x_n 
          = R_n(W_n x_{n-1} + V_n b_0)\;,
    \end{cases}		  
\end{equation}
where, for every $n\in \{1,\ldots,m\}$,
\begin{align}
&R_n = \prox_{\frac{\lambda_n \mu_n}{1+\lambda_n \chi_n} g_0}
\label{e:defRn}\\
&W_n =  \frac{1}{1+\lambda_n \chi_n}\left(\1 - \lambda_n \big(T^*T+D_n^*D_n)\right)
\label{e:defWn}\\
&V_n = \frac{\lambda_n}{1+\lambda_n \chi_n} \1\label{e:defVn}\\
&\chi_n = \mu_n \overline{\chi}\;, 
\end{align}
and
$\1$ denotes the identity operator.

From a theoretical standpoint, 
little is known about the theoretical properties of 
Model~\eqref{def:modelNN}
in relation with the minimization of the original regularized objective function \eqref{def:varJ}.
The main challenges
are that (i) the number of iterations $m$ is predefined, and (ii) the parameters can vary along the iterations.
%regularized objective function~\eqref{def:varJ}
%there is no guarantee that such a model constitutes a regularizing family, and there is no equivalence between the regularized inverse problem~\eqref{def:varJ} and the output of Model~\eqref{def:modelNN}, since the number of iterations $m$ is fixed in advance. 
However, insight can be gained by quantifying the robustness of Model~\eqref{def:modelNN} to perturbations on its initialization~$x_0$ and on its bias $b_0$, through an accurate estimation of its Lipschitz properties. 

\paragraph{Related works and contributions.}
There has been a plethora of techniques developed to invert models of the form 
\eqref{def:InversePb}.
%integrals of the form \eqref{def:T}.
Among these methods, Tikhonov-type methods are attractive from a theoretical viewpoint, especially because they provide good convergence rate as the noise level decreases, as shown in~\cite{hegland1995} or~\cite{hofmann2005}. 
Optimization techniques \cite{combettes2011proximal} are then classically used to solve Problem \eqref{def:varJ}.
However, limitations of such methods may be encountered in their implementation. 
Indeed, certain parameters such as gradient descent stepsizes or the regularization coefficient
need to be set, as discussed in~\cite{aakesson2008} or~\cite{daun2006}. 
The latter parameter depends on the noise level, as shown in~\cite{Engl1996}, which is not always easy to estimate.
In practical scenarios, methods like the L-curve method (see~\cite{hansen1999}) can be used to set the regularization parameter. \textcolor{black}{However, when dealing with more than one parameter, their application becomes computationally expensive.}
%they require a large number of resolutions and therefore a significant computational cost.
Moreover, incorporating constraints on the solution may be difficult in such approaches, and often reduces to projecting the resulting solution onto the desired set.
These reasons justify the use of a neural network structure to avoid laborious calibration of the parameters and to easily incorporate constraints on the solution.

The use of neural networks for solving inverse problems has become increasingly popular, 
especially in the image processing community.
A rich panel of approaches have been proposed, 
either adapted to the sparsity of the  data~\cite{antholzer2019sparse,kofler2018u}, 
or mimicking variational models~\cite{hammernik2018var,adler2017solving}, or iterating learned operators~\cite{aggarwal2018modl,meinhardt2017,pesquet2020learning,hasannasab2020,galinier2020},
or \textcolor{black}{adapting} Tikhonov method~\cite{li2020nett}.
% and/or using convolution neural network (CNN) whose parameters are trained on data related to the problem at hand~\cite{antholzer2019sparse,jin2017deep}.
 The successful numerical results of the aforementioned works raise two theoretical questions: in cases when these methods are based on the iterations of an optimization algorithm, do they converge (in the algorithmic sense)? Are these inversion methods stable or robust? %Can we topologically characterize the underlying or implicit regularization of such a method?

 In iterative approaches,
a regularization operator is learned, either in the form of a proximity (or denoiser) operator as~\cite{meinhardt2017,aggarwal2018modl,galinier2020}, of a regularization term~\cite{li2020nett},
of a pseudodifferential operator~\cite{bubba2021deep}, 
or of its gradient~\cite{gilton2019,Wu2020}. 
Strong connections also exist with Plug-and-Play methods~\cite{ryu2019plug,sun2019online,pesquet2020learning}, where the regularization operator is a pre-trained neural network.
Such strategies have in particular enabled high-quality imaging restoration or tomography inversion~\cite{Wu2020}. Here, the nonexpansiveness of the neural network is a core property to establish convergence of the algorithm~\cite{Wu2020,pesquet2020learning},
although recent works have provided convergence guarantees using alternative tools~\cite{Hurault2022}.
%But our proposed neural network is not based on this idea.

Other recent works solve linear inverse problems by unrolling the optimization iterative process in the form of a network architecture as in~\cite{Gregor2010,borgerding2016,jin2017deep}.
Here the number of iterations is fixed, instead of iterating until convergence, and the network is often trained in an end-to-end fashion.
Neural network frameworks offer powerful differential programming capabilities for learning hyper-parameters in an unrolled optimization algorithm as in~\cite{Corbineau2020,natarajan2020particle,Monga2021,Gharbi2024,Vy2024}.
\textcolor{black}{With respect to a traditional optimization method applied to \eqref{def:varJ}-\eqref{e:defJtau},
the use of these data-driven techniques provides an increased flexibility leading to a better solution of the inverse problem.}

All of the above strategies have shown very good numerical results.
However, few studies have been conducted on their theoretical properties, especially on their stability.
The study of the robustness of such a structure is often empirical, based on a series of numerical tests, as performed in~\cite{genzel2020robust,Vy2024}.
In~\cite{li2020nett}, stability properties are established at low noise regime when using neural networks regularization.
%they provide very large assumption under which the convergence and the regularization property of their network is ensured. But their result is not subject to verification during the numerical implementation.
A fine characterization of the convergence conditions of recurrent neural network and of their stability via the estimation of a Lipschitz constant is done in~\cite{Combettes2019,Combettes2020}. 
In particular, the Lipschitz constant estimated in~\cite{Combettes2020} is more accurate than in basic approaches 
which often rely in computing the product of the norms of the linear weight operators
of each layer as in~\cite{serrurier2020achieving,cisse2017parseval}.
Thanks to the aforementioned works, proofs of stability/convergence have been demonstrated on specific neural networks applied to classification (e.g., \cite{Neacsu2024, Neacsu2024-2}), or inverse problems (e.g.,~\cite{pesquet2020learning,hasannasab2020,Corbineau2020}). The analysis carried out in this article is in the line of these references.
%It is in the continuity of this type of results that this work is situated.
% In~\cite{li2020nett} and in~\cite{schwab2019deep}, Tikhonov regularization strategies are developed for an iterative neural network 
% and convergences properties of the network are shown, but little is known about the characterization of the limit point and 
% robustness properties.

Our contributions in this paper are the following. 
\begin{enumerate}
    \item We propose a neural network architecture to solve the inverse problem~\eqref{def:InversePb}, incorporating the proximity operator of a potentially nonsmooth convex function as its activation function. This design enables the imposition of constraints on the sought solution. A key advantage of this architecture lies in its foundation on the unrolling of a forward-backward algorithm, which ensures the network structure is both interpretable and frugal in terms of the number of parameters to be learned.
    \item We study theoretically and numerically the stability of the so-built neural network. The sensitivity analysis is conducted with respect to the observed data  $y$, which corresponds to a bias term in each layer of Model~\eqref{def:modelNN}.
This analysis is more general than the one performed in \cite{Corbineau2020}, in which only the impact of the initialization was considered.
%    \item We show how to implement the neural network in the case of Abel operators. Such operators arise in various physical applications. The proposed 
%    neural network performs numerically well compared to other classical inversion methods. Neural network techniques have been widely applied to imaging inverse problems, but few are tested on experimental 1D signal inverse problems.
\end{enumerate}

\paragraph{Outline.} 
The outline of the paper is as follows. In Section~\ref{section:notation}, we recall the theoretical background of our work and specify the notation, which follows the framework in~\cite{bauschke2011}.
Section~\ref{section:architecture} introduces a class of dynamical systems with a leakage factor, providing a more general framework than the neural network defined by~\eqref{def:modelNN}. 
 In Section~\ref{section:theory}, we establish the stability of the corresponding neural network defined on the product space 
$\X\times \X$, building upon the results from~\cite{Combettes2019} and~\cite{Combettes2020}.
By stability, we mean that the network output remains controlled with respect to both its initial input $x_0$
  and the bias term $T^*y$, ensuring that small differences or errors in these vectors are not amplified through the network. 
  We also provide sufficient conditions for the averagedness of the network. Finally, we show how stability properties for network \eqref{def:modelNN} can be deduced from these results.
Section~\ref{section:numerics} presents numerical illustrations of the derived Lipschitz bounds, followed by concluding remarks in Section~\ref{se:conclusion}.

\section{Notation}
\label{section:notation}

We introduce the elements of convex analysis
we will be dealing with, in Hilbert spaces. We also cover the bits of monotone operator theory that will be needed throughout. 

Let us consider the Hilbert space $\X$ 
endowed with the norm $\| \cdot\|$ and the scalar product $\langle \cdot, \cdot \rangle$. 
In the following, 
$\X$ shall refer to the underlying signal space.
%spaces of functions defined on the interval $]0,1[$ although the developments in Section \ref{section:theory} are valid for arbitrary Hilbert spaces.
The notation $\|\cdot\|$ will also refer to the operator norm of bounded operators from $\X$ onto $\X$.
%The identity operator over $\X$ will be referred to as $\1$.

An operator $S \colon \X \to \X $ is nonexpansive if it 
is $1-$Lipschitz, that is 
\[
(\forall (x,y) \, \in \, \X \times\X)
\qquad
\|Sx-Sy\| \leq \|x-y\| 
\; .
\]
Moreover, $S$ is said to be
%%%%%%%%%%%%%%%%%
\begin{enumerate}
    \item firmly nonexpansive if
\[
(\forall (x,y) \, \in \, \X \times\X)
\qquad
\| Sx -Sy\|^2
+
\| (\1-S)x -(\1-S)y\|^2
\leq \| x-y\|^2 
\; ;
    \]
    \item a Banach contraction if there exists $\kappa\in ]0,1[$ such that
\begin{equation}\label{e:strictcont}
(\forall (x,y) \, \in \, \X \times\X)
\qquad
\| Sx -Sy\| \le \kappa \| x-y\| 
\; .
\end{equation}
\end{enumerate}

If $S$ is a Banach contraction, 
then the iterates $(S^n x)_{n \in \N}$ converge linearly to a fixed point of $S$
according to Picard's theorem.
On the other hand, when $S$ is nonexpansive, 
the convergence is no longer guaranteed.
A way of recovering the convergence of the iterates 
is to assume that  $S$ is averaged,
i.e., 
there exists $\alpha \in ]0,1[$ 
and a nonexpansive operator $R:\X \to \X$ 
such that $S= (1-\alpha) \1+\alpha R$.
In particular, $S$ is $\alpha-$averaged
if and only if
\[
(\forall (x,y) \, \in \, \X \times\X)
\qquad
\| Sx -Sy\|^2
+
\frac{1-\alpha}{\alpha} \|(\1-S)x -(\1-S)y\|^2
\leq \| x-y\|^2 
\; .
\]
If $S$ has a fixed point and it is averaged, 
then the iterates $(S^n x)_{n \in \N}$ converge weakly to a fixed point.
Note that
$S$ is firmly nonexpansive if and only
if it is $1/2-$averaged and that, if $S$ satisfies \eqref{e:strictcont} with $\kappa \in ]0,1[$,
then it is $(\kappa+1)/2$-averaged.

Let $\Gamma_0(\X)$ be the set of proper lower semicontinuous convex functions from $\X$ to $]-\infty,+\infty]$.
Then we define the proximity operator \cite[Chapter 24]{bauschke2011} as
follows.
\begin{definition}
Let $f \in \Gamma_0(\X)$,
$x \in \X$,
and $\gamma>0$.
Then
$\prox_{\gamma f} (x)$ is the unique point that satisfies
\[
\prox_{\gamma f} (x) \,
= \, 
\underset{y \in \X }{\operatorname{argmin}} \, \left(
f(y) + \frac{1}{2\gamma} \| x-y \|^2
\right)
\; .
\]
The function $\prox_{\gamma f} : \X \to \X $ is the proximity operator of $\gamma f$.
\end{definition}
Finally, the proximity operator has the following property.
\begin{proposition}{\rm \cite[Proposition 12.28]{bauschke2011}}
The operators
$\prox_{\gamma f}$ and $\1-\prox_{\gamma f}$
are firmly nonexpansive.
\end{proposition}

In the proposed neural network~\eqref{def:modelNN}, 
the activation operator is a proximity operator. 
In practice, this is the case for most activation operators, as shown in~\cite{Combettes2019}.
The neural network~\eqref{def:modelNN} is thus a cascade of firmly nonexpansive operators and linear operators. 
If the linear part is also nonexpansive, bounds on the effect of a perturbation of the neural network or its iterates can be established.

%%%%%%%%%%%%%%%%%%%%%%%%%%%%%%%%%%%%%%%%%%%%%%%%%%%%%%%%%%%%%%%%%%%%%%%%%%%%%%%%%%%%%%%%%%%%%%%%%%%%%%%%%%%%%%%%%%%%%%%%%%%%%%%%%
\section{Considered neural network architecture}
\label{section:architecture}
%
%%%%%%%%%%%%%%%%%%%%%%%%%%%%%%%%%%%%%%%%%%
In this section, 
Model~\eqref{def:modelNN}
is reformulated as a virtual network, which takes as inputs the classical ones on top of a new one, which is a bias parameter.

\subsection{Virtual neural network with leakage factor}\label{se:VNN1}
To facilitate our theoretical analysis, we will introduce a virtual network making use of new variables $(z_n)_{n\in \mathbb{N}}$ in the product space $\X \times \X$. 
For every $n \in \mathbb{N}\setminus\{0\}$,
we define the $n$-th layer of our virtual network as follows:
\begin{equation}
    \label{def:nn-virtual-bias}
%    (\forall n \in \mathbb{N})\quad
z_n = \left( \begin{array}{c}
 x_n \\
 b_n
\end{array}
\right) ,
\quad
z_{n} = Q_n(U_n z_{n-1})
\; ,
\quad
\text{with}
\quad
\begin{cases}
\displaystyle Q_n = \left( \begin{array}{c}
  R_n   \\
  \1
\end{array} \right)
\; , \\
\\
\displaystyle U_n = \left( \begin{array}{cc}
   W_n  &  %\displaystyle 
   V_n
   %\frac{\lambda_n}{1+\lambda_n\chi_n} \1 
   \\
   0    & \eta_n \1
\end{array}\right)
\;,
\end{cases}
\end{equation}
and $W_n$ (resp. $V_n$) defined by
\eqref{e:defWn} (resp. \eqref{e:defVn}).
In our context,
$R_n$ is defined by \eqref{e:defRn}, but all the results in our stability analysis remain
valid for any firmly nonexpansive operator such as the resolvent of a maximally monotone operator.
%\begin{align}
%&R_n = \text{prox}_{\lambda_n \mu_n g}\\
%&W_n =  \1 - \lambda_n T^*T - \lambda_n %\tau_n D^*D.
%\end{align}
%where $\1$ denotes the identity operator.
Note that, in order to gain more flexibility, we have 
 included positive multiplicative factors $(\eta_n)_{n\geq 1}$ on the bias.
 \textcolor{black}{Since, for every $n\ge 1$, $b_n = \eta_n\ldots\eta_1 b_0$,}
 cascading $m$ such layers yields
  \begin{equation}
	\label{def:modelNN-leakage}
	\begin{cases}
	\textbf{Initialization:} \\
	\quad b_0 = T^*y ,\\
	\textbf{Layer $n\in \{1,\ldots,m\}$:} \\
      \quad x_n %=  Q_n(x_0,x_{n-1}) 
          = R_n(W_n x_{n-1} + \widetilde{V}_n b_0)\\
          \qquad\;\, = \prox_{\lambda_n \mu_n g}\big(x_{n-1}- \lambda_n (T^*T+D_n^*D_n)x_{n-1} + \lambda_n \eta_{n-1}\cdots \eta_0 b_0\big)      
          \;,
    \end{cases}		  
\end{equation}
where 
\begin{equation}
%&R_n = \text{prox}_{\lambda_n \mu_n g}\\
%&W_n =  \1 - \lambda_n T^*T - \lambda_n \tau_n D^*D \\
%&
\widetilde{V}_n = \textcolor{black}{\frac{\lambda_n}{1+\lambda_n \chi_n}} \eta_{n-1}\cdots \eta_0 \1 \quad \text{and}\quad \eta_{0}=1\;.
\end{equation}
%and $\eta_{0}=1$.
We thus see that the  network defined by Model~\eqref{def:modelNN} 
is equivalent to the virtual one when all the factors $\eta_{n}$ are equal to one.
When $n\ge 1$ and $\eta_{n}< 1$,
the  parameters $\eta_{n}$
can be interpreted as a leakage factor. 
\begin{remarque}
In the original forward-backward algorithm, the introduction of $(\eta_{n})_{n\geq 1}$ amounts to introducing an error $e_n$ in the gradient step, at iteration $n$, which is equal to
\begin{equation}
    e_n = \textcolor{black}{\frac{\lambda_n}{1+\lambda_n \chi_n}} (\eta_{n-1}\cdots \eta_0-1) b_0.
\end{equation}
From known properties 
concerning the forward-backward algorithm~{\rm \cite{combettes2011proximal}}, 
the standard convergence results for the algorithm are still valid provided that
\begin{equation}
    \sum_{n=2}^{+\infty} \textcolor{black}{\frac{\lambda_n}{1+\lambda_n \chi_n}} |\eta_{n-1}\cdots \eta_0-1| < +\infty
    \; .
\end{equation}
\end{remarque}

\subsection{Multilayer architecture}\label{se:VNN1m}

When we cascade the layers of the virtual neural network~\eqref{def:nn-virtual-bias},
the following triangular linear operator plays a prominent role:
%£In our subsequent analysis, it will be useful to define the triangular linear operator
\begin{equation}
    \label{def:U-bias}
U = \;U_m \circ \cdots \circ U_1 
  = \;
\left( \begin{array}{cc}
  W_{1,m}   &  \widetilde{W}_{1,m}\\
  0   &  \eta_{1,m} \1
\end{array} \right)
\; ,
\end{equation}
where, for every  $n\in \{1,\ldots,m\}$ and $i\in \{1,\ldots,n\}$ 
\begin{equation}\label{e:defWtin}
\widetilde{W}_{i,n} =  
\sum_{j=i}^{n} \frac{\lambda_{j}}{1+\lambda_j \chi_j} \eta_{i,j-1}W_{j+1,n}\, 
\end{equation}
%\begin{align}
%\widetilde{W} &=   
%  \sum_{n=1}^{m}  \lambda_{n} \eta_{1,n-1} W_{n+1,m} 
%  \end{align}
and, for every $i\in \{1,\ldots,m+1\}$
and  $j\in \{0,\ldots,m\}$,
\begin{align}
&   W_{i,j} = 
\begin{cases}
W_j\circ \cdots \circ W_i & \mbox{if $j\ge i$}\\
\1 & \mbox{otherwise,}
\end{cases}\label{def:WjWi}\\
&    \eta_{i,j} = 
    \begin{cases}
    \eta_{j}\cdots \eta_{i} & \mbox{if $j\ge i$}\\
    1 & \mbox{otherwise.}
    \end{cases}\label{def:N-bias}
\end{align}
In particular, $W_{i,i} = W_i$ and $\eta_{i,i} = \eta_i$.

\section{Stability and $\alpha$-averagedness}
\label{section:theory}
%
%%%%%%%%%%%%%%%%%%%%%%%%%%%%%%%%%%%%%%%%%%
%In this section, we study the stability of the proposed neural network~\eqref{def:modelNN}. 
%This analysis is performed by estimating the Lipschitz constant of the network, and by determining under which conditions  $\alpha$-averagedness is achieved. 

\textcolor{black}{
This section presents our stability results. Section~\ref{se:assume} introduces the main assumptions and defines key quantities relevant to our analysis. Section~\ref{se:stabresVN} focuses on the stability analysis of the Virtual Neural Network defined in~\eqref{def:nn-virtual-bias}, considered as an operator $S$ from $\X \times \X$ to itself, with input $(x_0,b_0)$ and output $(x_m,b_m)$. The core result in Proposition~\ref{p:LipVNN-bias} provides an estimate $\theta_m$ of its Lipschitz constant, which can be computed with quadratic complexity. Additionally, we derive simpler upper and lower bounds for the Lipschitz constant. Proposition~\ref{prop:ab-bias} offers sufficient conditions for the VNN to be averaged. Notably, in the absence of a leakage factor, the VNN cannot be nonexpansive, as shown in Proposition~\ref{prop:leakage}. 
First estimates of the Lipschitz constant for network~\eqref{def:modelNN} are derived in Section~\ref{se:directlink}. Section~\ref{se:x0b0xm} aims to refine the previous Lipschitz  bounds by focusing on the first output of the VNN, so coming back to the original network~\eqref{def:modelNN}. This enables, in Section~\ref{se:netsingisingo}, an investigation of the stability of the network when its initialization $x_0$ is linked to the observation $y$ through a filtering relation—thus defining an operator from $\X$ to $\X$ (see Proposition~\ref{p:LipVNN-proj2}).
}

\subsection{Assumptions}\label{se:assume}
We will make the following assumption on the degradation operator $T$ and the regularization operators $(D_n)_{1\le n \le m}$:\\
\fbox{
Operators $T^*T$ and  $(D_n^*D_n)_{1\le n \le m}$
can be diagonalized
in an orthonormal basis
$(v_{p})_{p}$ of $\X$.
%in the same orthonormal set of eigenvectors $(v_{p})_{p}$.
}

The existence of such an eigenbasis is satisfied in the following cases:
\begin{itemize}
\item $T$ and $(D_n)_{1\le n \le m}$ are compact operators (which guarantees the existence of their singular value decompositions),
and $(D_n^*D_n)_{0\le n \le m}$ commute pairwise if we set
$D_0 = T$.
\item $\X=\Y$ is the space $\mathbb{L}^2([0,\mathcal{T}])$ of square summable complex-valued functions defined on $[0,\mathcal{T}]$,
$T$ is a filter (i.e., circular convolutive operator) and, for every $n\in \{1,\ldots,m\}$, $D_n$ is a single-input $c$-output  filter that maps every signal $x\in \X$ to a vector $(D_{n,1}x,\ldots,D_{n,c} x) \in \X^c$. In this case, 
the diagonalization is performed in the Fourier basis defined as
\[
(\forall p \in \mathbb{Z})
(\forall t \in [0,\mathcal{T}])
\quad v_p(t) = \frac{1}{\sqrt{\mathcal{T}}}\exp\left(2\pi\imath \frac{p t}{\mathcal{T}}\right).
\]
\item $\X=\Y$ is the space $\ell^2(N)$
of complex-valued sequences $x=(x_n)_{0\le n \le N-1}$ equipped with the standard Hermitian norm, $T$ is a discrete periodic filter
and, for every $n\in \{1,\ldots,m\}$, $D_n$ is a single-input $c$-output  periodic filter that maps every signal $x\in \X$ to a vector $(D_{n,1}x,\ldots,D_{n,c} x) \in \X^c$. The eigenbasis is associated with the discrete Fourier transform:
\[
(\forall (p,n) \in \{0,\ldots,N-1\}^2)
\quad v_p(n) = \frac{1}{\sqrt{N}}\exp\left(2\pi\imath \frac{pn}{N}\right).
\]
\item The two previous examples extend to
$d$-dimensional signals ($d=2$ for images),
defined on $[0,\mathcal{T}_1]\times\cdots\times [0,\mathcal{T}_d]$
or $\{0,\ldots,N_1-1\}\times\cdots\times \{0,\ldots,N_d-1\}$.
\end{itemize}

Based on the above assumption,
we define the respective nonnegative real-valued eigenvalues
$(\beta_{T,p})_p$ and $(\beta_{D_n,p})_p$ of $T^*T$ and $D_n^*D_n$ with $n\in\{1,\ldots,m\}$, as well as 
the following quantities, 
for every eigenspace index $p$ and $i\in \{1,\ldots,n\}$,
\begin{align}
\label{def:vp0}
     &\beta_p^{(n)} = \frac{1}{1+\lambda_n\chi_n}\Big(1 - \lambda_n (\beta_{T,p} +\beta_{D_n,p})\Big) \; ,\\
\label{def:vp1}
     &\beta_{i,n,p} = \prod_{j=i}^n \beta_p^{(j)} \; , \\ 
\label{def:vp-bias}
     &  \widetilde{\beta}_{i,n,p} =   \sum_{j=i}^{n-1} \beta_p^{(n)} \cdots \beta_p^{(j+1)} \frac{\lambda_{j}\eta_{i,j-1}}{1+\lambda_j\chi_j}
     + \frac{\lambda_n \eta_{i,n-1}}{1+\lambda_n\chi_n}
     \;, %\label{def:betatinp}
\end{align}
with the convention $\sum_{i=n}^{n-1} \cdot = 0$.
As limit cases, we have
\begin{equation}\label{e:limitbetann}
\beta_{n,n,p} = \beta_p^{(n)}, \quad \widetilde{\beta}_{n,n,p} = \frac{\lambda_n}{1+\lambda_n\chi_n}.
\end{equation}
Note that $(\beta_p^{(n)},v_{p})_{p}$,
$(\beta_{i,n,p},v_p)_{p}$,
and $(\widetilde{\beta}_{i,n,p},v_p)_{p}$
are the eigensystems of $W_{n}$,
$W_{i,n}$ and~$\widetilde{W}_{i,n}$, 
defined by \eqref{e:defWn}, \eqref{def:WjWi}, and \eqref{e:defWtin}, 
respectively.

\subsection{Stability results for the virtual network}\label{se:stabresVN}
We first recall some existing results on the stability of neural networks \cite[Proposition~3.6(iii)]{Combettes2019} \cite[Theorem 4.2]{Combettes2020}.
\begin{proposition}\label{prop:oldres}
Let $m > 1$ be an integer, let $(\mathcal{H}_i)_{0 \leq i \leq m}$ be non null Hilbert spaces.
For every $n \in \{1,\ldots,m\}$, let $U_n$ be a bounded linear operator from $\mathcal{H}_{n-1}$ to $\mathcal{H}_{n}$
and let $Q_n\colon \mathcal{H}_n \to \mathcal{H}_n$ be a firmly nonexpansive operator.
Set $U = U_{m} \circ \cdots \circ U_1$ and
\begin{multline}
\label{e:defthetaell}
\theta_m=\|U\|\\
+\sum_{k=1}^{m-1}\sum_{1\leq j_1<\ldots<j_k\leq m-1}
\|U_m\circ\cdots\circ U_{j_k+1}\|\,
\|U_{j_k}\circ\cdots\circ U_{j_{k-1}+1}\|\cdots 
\|U_{j_1}\circ\cdots\circ U_1\|.
\end{multline}
Let $S = Q_m\circ U_m \circ \cdots \circ Q_1 \circ U_1$.
Then the following hold:
\begin{enumerate}
    \item\label{prop:oldresi} $\theta_m/2^{m-1}$ is a Lipschitz constant of $S$.
    \item\label{prop:oldresii} 
    Let $\alpha \in [1/2, 1]$.
    If $\mathcal{H}_m = \mathcal{H}_0$ and 
    \begin{equation}\label{e:condalphaold}
\| U - 2^{m} (1- \alpha) \1 \|
- \|U\|
+2 \theta_{m}
\leq
2^{m} \alpha
\;,
\end{equation}
then $S$ is $\alpha$-averaged.
\end{enumerate}
\end{proposition}
In light of these results, 
we will now analyze the properties of the virtual network~\eqref{def:nn-virtual-bias}
based on the spectral quantities
defined at the end of Section \ref{se:assume}, 
%defined by~\eqref{def:T} and~\eqref{def:D},
and the parameters $(\lambda_n)_{1\le n\le m}$, $\overline{\chi}$ (and possibly $(\tau_n)_{1\le n \le m}$). 
One of the main difficulties with respect to the case already studied by~\cite{Corbineau2020}
is that, in our context, the involved operators $(U_n)_{1\le n \le m}$ are no longer self-adjoint.

A preliminary result will be needed:
\begin{lemme}\label{e:normUni-bias}
Let $m \in \mathbb{N}\setminus \{0\}$ be 
the total number of layers.
For every layer indices $n\in \{1,\ldots,m\}$ and 
$i\in \{1,\ldots,n\}$, 
the norm of $U_n\circ \cdots \circ U_i$ is 
equal to $\sqrt{a_{i,n}}$ with
\begin{equation}\label{e:defain-bias}
a_{i,n} = \frac{1}{2} \sup_{p}
\left(
\beta_{i,n,p}^2 + \widetilde{\beta}_{i,n,p}^2 + \eta_{i,n}^2+ \sqrt{(\beta_{i,n,p}^2  + \widetilde{\beta}_{i,n,p}^2  +  \eta_{i,n}^2 )^2
        - 4\beta_{i,n,p}^2 \eta_{i,n}^2}
\right) \; ,
\end{equation}
where $p$ indexes the eigenspaces of $T^*T$. %defined by~\eqref{def:T}.
\end{lemme}
\begin{proof}
Thanks to expressions \eqref{def:nn-virtual-bias}, \eqref{e:defWtin}, \eqref{def:WjWi}, and \eqref{def:N-bias},
%~\eqref{def:vp1}, \eqref{def:N-bias} and~\eqref{def:vp-bias}, 
we can calculate the norm of $ U_n\circ \cdots \circ U_i$.
For every $z =(x,b)$, $U_n\circ \cdots \circ U_i z = (W_{i,n} x + \widetilde{W}_{i,n} b, \eta_{i,n} b)$ and
\[
\begin{split}
\|U_n\circ \cdots \circ U_i z \|^2 = & \; \| W_{i,n} x + \widetilde{W}_{i,n} b\|^2 + \eta_{i,n}^2 \|b\|^2\\
= & 
\; \| W_{i,n}x\|^2 + 2\langle W_{i,n} x , \widetilde{W}_{i,n} b \rangle 
+ \| \widetilde{W}_{i,n} b \|^2 + \eta_{i,n}^2 \|b\|^2
\; .
\end{split}
\]
%where 
%\begin{equation}
%\widetilde{W}_{i,n} =  
%\sum_{j=i}^{n} \lambda_{j} \eta_{i,j-1}W_{j+1,n}\, .
%\end{equation}
Let $(\beta_{i,n,p},v_p)_{p\in\mathbb{N}}$ 
defined by~\eqref{def:vp1}
and $(\widetilde{\beta}_{i,n,p},v_p)_{p\in\mathbb{N}}$
defined by~\eqref{def:vp-bias}
be the respective eigensystems of~$W_{i,n}$ and~$\widetilde{W}_{i,n}$.
Let us decompose $(x,b)$ on the basis of eigenvectors $(v_{p})_{p}$ of $T^*T$, % defined by~\eqref{def:T},
as 
\begin{equation}\label{e:decxbbasis}
\begin{cases}
\displaystyle x = \sum_p \xp \, v_p \; , \\
\displaystyle b = \sum_p \bp \, v_p \; .
\end{cases} 
\end{equation}
In the following, we will assume that $\X$ is a real Hilbert space. A similar proof can be followed for complex Hilbert spaces.
We have then
\[
\|U_n\circ \cdots \circ U_i z\|^2 = \; \sum_p \beta_{i,n,p}^2 \xp^2 + 2 \sum_p \beta_{i,n,p} \widetilde{\beta}_{i,n,p} \xp \bp
+ \sum_p (\widetilde{\beta}_{i,n,p}^2 + \eta_{i,n}^2) \bp^2
\; .
\]
By definition of the operator norm,
\[
 \| U_n\circ \cdots \circ U_i\|^2 = \; \underset{\|x\|^2 +\|b\|^2=1}{\sup}
\left(
\sum_p \beta_{i,n,p}^2 \, \xp^2 
+ (\eta_{i,n}^2+\widetilde{\beta}_{i,n,p}^2) \, \bp^2 
+ 2\beta_{i,n,p} \widetilde{\beta}_{i,n,p} \, \xp  \,\bp \right)\; .
\]
Note that, for every integer $p$ and $\zp = (\xp, \bp)\in \mathbb{R}^2$,
\begin{equation}
\beta_{i,n,p}^2 \, \xp^2 +
(\eta_{i,n}^2+\widetilde{\beta}_{i,n,p}^2) \, \bp^2 
+ 2\beta_{i,n,p} \widetilde{\beta}_{i,n,p} \, \xp \, \bp 
= \zp^\top
A_{i,n,p}\, \zp
\end{equation}
where 
%$\langle \cdot, \cdot \rangle$ denotes the Euclidean inner product and 
\[A_{i,n,p} = \left( 
\begin{array}{cc}
  \beta_{i,n,p}^2   & \textcolor{black}{\beta_{i,n,p} \widetilde{\beta}_{i,n,p}} \\
  \widetilde{\beta}_{i,n,p}\beta_{i,n,p}     & \eta_{i,n}^2 + \widetilde{\beta}_{i,n,p}^2
\end{array}
\right) \; .
\]
Hence,
\begin{equation} \label{e:defnormU2-bias-before}
 \| U_n\circ \cdots \circ U_i \|^2 = \; \underset{z = (\zp)_p, \|z\| =1}{\sup}
             \sum_p \zp^\top
A_{i,n,p}\, \zp 
\; .
\end{equation}
\textcolor{black}{
For every index $p$, let $\nu_{i,n,p}$ be  the maximum eigenvalue of the symmetric positive semidefinite matrix $A_{i,n,p}$.
We have
\begin{equation*}
 \sum_p \zp^\top
A_{i,n,p}\, \zp 
\le \sum_p \nu_{i,n,p}
\| \omega_p\|^2
\le \left(\sup_p\, \nu_{i,n,p}\right)
\sum_{p} \|\omega_p\|^2
= \left(\sup_p\, \nu_{i,n,p}\right)
\|z\|^2\;.
\end{equation*}
Altogether with \eqref{e:defnormU2-bias-before},
this shows that
\begin{equation*}
\| U_n\circ \cdots \circ U_i \|^2 
\le \sup_p\, \nu_{i,n,p}\;.
\end{equation*}
In addition, from the definition of 
the supremum, for every $\epsilon >0$,
there exists $p^*$ such that
\[
 \sup_p\, \nu_{i,n,p}-\epsilon < \nu_{i,n,p^*}.
\]
By setting $z^* = (\omega^*_p)_p$ where
$\omega^*_p = 0$ if
$p\neq p^*$, and
$\omega^*_{p^*}$ is a unit norm eigenvector 
associated with eigenvalue $\nu_{i,n,p^*}$,
\begin{equation*}
 \sum_p \zp^\top
A_{i,n,p}\, \zp
= \omega_{p^*}^\top
A_{i,n,p^*}\, \omega_{p^*}
= \nu_{i,n,p^*} 
>  \sup_p\, \nu_{i,n,p}-\epsilon\;.
\end{equation*}
Since $\epsilon$ can be chosen arbitrarily small, we deduce that
%Since $A_{i,n,p}$ is a symmetric positive semidefinite matrix, 
\begin{equation}\label{e:defnormU2-bias}
\| U_n\circ \cdots \circ U_i \|^2 = 
%\underset{p, \|\zp\| =1}{\sup} \zp^\top
%A_{i,n,p}\, \zp = 
\sup_p\, \nu_{i,n,p}\; .
\end{equation}
%where, for every index $p$, $\nu_{i,n,p}$ is  the maximum eigenvalue of $A_{i,n,p}$.
We finally note that
the two eigenvalues of matrix $A_{i,n,p}$ are the roots of the characteristic polynomial 
}
\[
\begin{split}
(\forall \nu \in \mathbb{R})\quad
    \text{det} (A_{i,n,p} - \nu \1_2) = & \;
(\beta_{i,n,p}^2 - \nu)(\widetilde{\beta}_{i,n,p}^2 + \eta_{i,n}^2 - \nu) - \beta_{i,n,p}^2\widetilde{\beta}_{i,n,p}^2 \\
                           = & \;
\nu^2 - (\beta_{i,n,p}^2 + \widetilde{\beta}_{i,n,p}^2 + \eta_{i,n}^2 )\nu 
+ \beta_{i,n,p}^2\eta_{i,n}^2
\; .
\end{split}
\]
The discriminant of this second-order polynomial reads
\[
\begin{split}
\Delta_{i,n,p} = & \;  (\beta_{i,n,p}^2 + \widetilde{\beta}_{i,n,p}^2 + \eta_{i,n}^2 )^2 - 4\beta_{i,n,p}^2 \eta_{i,n}^2\\
        =& \;  (\beta_{i,n,p}^2 -\widetilde{\beta}_{i,n,p}^2 - \eta_{i,n}^2)^2
       +4 \beta_{i,n,p}^2\widetilde{\beta}_{i,n,p}^2 \; \ge 0
        \; .
\end{split}
\]
Therefore, for every integer $p$,
\begin{equation}\label{e:defnuinp-bias}
\nu_{i,n,p} = \frac{1}{2} 
\left(
\beta_{i,n,p}^2 + \widetilde{\beta}_{i,n,p}^2 + \eta_{i,n}^2 + \sqrt{(\beta_{i,n,p}^2  + \widetilde{\beta}_{i,n,p}^2  + \eta_{i,n}^2 )^2
        - 4\beta_{i,n,p}^2 \eta_{i,n}^2}
\right) \; .
\end{equation}
%%By going back to \eqref{e:defnormU2-bias},
%we have
%\[
%\| U_n\circ \cdots \circ U_i  \|^2
%= 
%\sup_{
%\substack{(\omega_p)_p\\
%\sum_p \|\omega_p\|^2= 1}}
%\|\omega_p\|^2 \nu_{i,n,p}
%\le a_{i,n}.
%\]
%In addition, from the definition of $a_{i,n}$, for every $\epsilon >0$,
%there exists $p^*$ such that
%\[
%a_{i,n}-\epsilon < \nu_{i,n,p^*},
%\]
%which shows that
%\[
%a_{i,n}-\epsilon < 
%\| U_n\circ \cdots \circ U_i  \|^2
%\le a_{i,n}.
%\]
%Going back to ,
%We deduce that
\textcolor{black}{
It then follows from \eqref{e:defnormU2-bias} that
}
\[
\| U_n\circ \cdots \circ U_i  \|^2 
= a_{i,n} \; .
\]
\end{proof}

\begin{remarque}\label{re:weightednorm}
The previous result assumes that the product space $\X\times \X$ is equipped with the standard norm. Since we may be interested in the stability w.r.t. variations of the observed data $y$, and $b_0 = T^* y$, it might be more insightful to consider instead the following weighted norm:
\begin{equation}
\label{e:normweight}
(\forall z = (x,b)\in \X\times \X)\quad
\|z\| = \sqrt{\|x\|^2+\|b\|_{\varpi}^2}\,,
\end{equation}
where $b$ is decomposed as in \eqref{e:decxbbasis}, 
\begin{equation}
\|b\|_{\varpi}^2 = \sum_p \frac{|\zeta_p|^2}{\varpi_p}\, ,
\end{equation}
and \textcolor{black}{$\varpi =(\varpi_p)_p$} is such that $\inf_p \varpi_p > 0$
and $\sup_p \varpi_p < +\infty$. The latter weights aim to compensate the effect of $T^*$ on the observed data.
This space renormalization is possible since
the operator $Q_n$ remains firmly nonexpansive after this norm change, because of its specific structure.\\ 
\textcolor{black}{As shown in Appendix \ref{secA0},}
the expression of $a_{i,n}$
for $n\in \{1,\ldots,m\}$
and $i\in \{1,\ldots,n\}$
in Lemma~\ref{e:normUni-bias} is then modified as follows:
\begin{equation}\label{e:defain-bias-weighted}
a_{i,n} = \frac{1}{2} \sup_{p}
\left(
\beta_{i,n,p}^2 + 
\varpi_p
\widetilde{\beta}_{i,n,p}^2 + \eta_{i,n}^2+ \sqrt{(\beta_{i,n,p}^2  + \varpi_p\widetilde{\beta}_{i,n,p}^2  +  \eta_{i,n}^2 )^2
        - 4\beta_{i,n,p}^2 \eta_{i,n}^2}
\right) \; .
\end{equation}
\end{remarque}

We will now quantify the Lipschitz regularity of the network.
\begin{proposition}\label{p:LipVNN-bias}
Let $m \in \mathbb{N}\setminus \{0\}$. 
For every $n\in \{1,\ldots,m\}$ and 
$i\in \{1,\ldots,n\}$,  let $a_{i,n}$ be given by
\eqref{e:defain-bias}.
Set $\theta_0 = 1$ and
define $(\theta_n)_{1\le n \le m}$ recursively by
\[
(\forall n \in \{1,\ldots,m\})\quad 
\theta_n = 
\sum_{i=1}^n \theta_{i-1} \sqrt{a_{i,n}}\;.
\]
Then 
$\theta_m/2^{m-1}$ is a Lipschitz constant of the virtual network~\eqref{def:nn-virtual-bias}.
In addition,
\begin{equation}\label{e:looseLipb}
\sqrt{a_{1,m}}\le \frac{\theta_m}{2^{m-1}}\le  \Big(\prod_{n=1}^m a_{n,n}\Big)^{1/2}.
\end{equation}
\end{proposition}
\begin{proof}
According to Proposition~\ref{prop:oldres}\ref{prop:oldresi}, 
if $\theta_m$ is given by \eqref{e:defthetaell}, then $\theta_m/2^{m-1}$ is a Lipschitz constant of the virtual network~\eqref{def:nn-virtual-bias}. 
On the other hand, it follows from
\cite[Lemma 3.3]{Combettes2019} that $\theta_m$ can be calculated recursively as
\[
(\forall n \in \{1,\ldots,m\})\quad 
\theta_n = 
\sum_{i=1}^n \theta_{i-1}  \| U_n \circ \cdots \circ U_i \|\;,
\]
with  $\theta_0 = 1$. Finally, Lemma \ref{e:normUni-bias} allows us to substitute $(\sqrt{a_{i,n}})_{1\le i \le n}$ for $(\| U_n \circ \cdots \circ U_i \|)_{1\le i \le n}$ in the above expression.\\
In addition, according to \cite[Proposition 4.3(i)]{Combettes2020},
\[
\|U_m \circ \cdots \circ U_1\|\le \frac{\theta_m}{2^{m-1}} \le
\prod_{n=1}^m \|U_n\|.
\]
It follows from Lemma \ref{e:normUni-bias},
that $\|U_m \circ \cdots \circ U_1\| = \sqrt{a_{1,m}}$ and
\[
(\forall n \in \{1,\ldots,m\})\quad
\|U_n\| = \sqrt{a_{n,n}},
\]
which allows us to deduce inequality \eqref{e:looseLipb}.
\end{proof}

\begin{remarque}
    Proposition \ref{p:LipVNN-bias}
    shows that the complexity for computing the proposed Lipschitz bound is quadratic as a function 
    of the number of layers $m$
    (more precisely, $O(m(m+1)/2))$.
\end{remarque}

We can relate the bounds provided on the Lipschitz constant in the previous proposition to simpler ones.
\begin{corollaire}
Assume that $\X \times \X$ is equipped with the norm defined by \eqref{e:normweight}.
Then
\begin{multline}\label{e:looserLipb}
\sup_p
\max\left\{\Big|\prod_{n=1}^m \beta_p^{(n)}\Big|,
\sqrt{\varpi_p\widetilde{\beta}_{1,m,p}^2+\eta_{1,m}^2}\right\}
\le  \frac{\theta_m}{2^{m-1}} \\ \le 
\prod_{n=1}^m 
\left(\sup_p\,
\left( (\beta_p^{(n)})^2 
+ \varpi_p \Big(\frac{\lambda_n}{1+\lambda_n\chi_n}\Big)^2
+\eta_n^2\right)\right)^{1/2}.
\end{multline}
\end{corollaire}

\begin{proof}
\textcolor{black}{
For every $n\in \{1,\ldots,m\}$ and integer $p$,
\[
\sqrt{(\beta_{n,n,p}^2  + \varpi_p\widetilde{\beta}_{n,n,p}^2  +  \eta_{n,n}^2 )^2
        - 4\beta_{n,n,p}^2 \eta_{n,n}^2}
\le \beta_{n,n,p}^2  + \varpi_p\widetilde{\beta}_{n,n,p}^2  +  \eta_{n,n}^2\;.
\]
We deduce from \eqref{e:defain-bias-weighted} that}
\begin{align*}
    a_{n,n} &\le \sup_p \beta_{n,n,p}^2 + 
\varpi_p
\widetilde{\beta}_{n,n,p}^2 + \eta_{n,n}^2\\
& = \sup_p\,(\beta_p^{(n)})^2 + 
\varpi_p
\Big(\frac{\lambda_n}{1+\lambda_n\chi_n}\Big)^2 + \eta_n^2\ ,
\end{align*}
\textcolor{black}{
where the last equality follows
from \eqref{def:N-bias}
and \eqref{e:limitbetann}.
}
The majorization in \eqref{e:looserLipb} is then a consequence of the upper bound in 
\eqref{e:looseLipb}.\\
On the other hand, \textcolor{black}{by rewriting \eqref{e:defain-bias-weighted} in a slightly different way,}
\begin{align*}\label{e:defain-bias-weighted}
a_{1,m} = &\frac{1}{2} \sup_{p}
\left(
\beta_{1,n,p}^2 + 
\varpi_p
\widetilde{\beta}_{1,m,p}^2 + \eta_{1,m}^2\right.\\
&\qquad\quad\left.+ \sqrt{(\beta_{1,m,p}^2  - \varpi_p\widetilde{\beta}_{1,m,p}^2 -  \eta_{1,m}^2 )^2
        + 4\varpi_p\beta_{1,m,p}^2 \widetilde{\beta}_{1,m,p}^2}
\right) \; ,\\
\ge & \frac{1}{2} \sup_{p}
\left(
\beta_{1,n,p}^2 + 
\varpi_p
\widetilde{\beta}_{1,m,p}^2 +\eta_{1,m}^2+
|\beta_{1,m,p}^2  - \varpi_p\widetilde{\beta}_{1,m,p}^2 -  \eta_{1,m}^2|\right)\\
= & \sup_p
\max\Big\{
\beta_{1,n,p}^2,
\varpi_p\widetilde{\beta}_{1,m,p}^2+\eta_{1,m}^2\}.
\end{align*}
\end{proof}

\begin{remarque}
As shown in {\rm \cite{Corbineau2020}}, if we are just interested in the sensitivity w.r.t. $x_0$ of the unfolded algorithm by assuming that $b_0$ is unperturbed, formulas in 
Proposition~\ref{p:LipVNN-bias} remain valid 
by setting  
\[
(\forall n \in \{1,\ldots,m\})
(\forall i \in \{1,\ldots,n\})\quad
a_{i,n} = \sup_p |\beta_{i,n,p}|.
\]
This can be viewed as a limit case of \eqref{e:defain-bias-weighted} where 
$\sup_p \varpi_p \to 0$ and, for every $n\in \{1,\ldots,m\}$, $\eta_n = 0$.
\end{remarque}

We will next provide conditions ensuring that the virtual network is an averaged operator.
\begin{proposition}
\label{prop:ab-bias}
Assume that $\X \times \X$ is equipped with the norm defined by \eqref{e:normweight}.
Let $m \in \mathbb{N}\setminus \{0\}$.
Let $a_{1,m}$ be defined by \eqref{e:defain-bias-weighted} and $\theta_m$ be defined in Proposition \ref{p:LipVNN-bias}.
Let $\alpha \in [1/2,1]$.
Define
\begin{align}
b_\alpha =  \frac{1}{2} \sup_p &\;
\left(
 (\beta_{1,m,p}- \gamma_\alpha)^2  +  (\eta_{1,m}- \gamma_\alpha)^2 +\varpi_p\widetilde{\beta}_{1,m,p}^2 \right. \nonumber\\
& \; \left. + \; \sqrt{
\begin{array}{c}
\big( (\beta_{1,m,p}- \gamma_\alpha)^2  +  (\eta_{1,m}-\gamma_\alpha)^2 + \varpi_p\widetilde{\beta}_{1,m,p}^2 \big)^2 \\
 - 4  (\beta_{1,m,p}- \gamma_\alpha)^2(\eta_{1,m}- \gamma_\alpha)^2
\end{array}
}
\; \right) \; ,
\end{align} 
with $\gamma_\alpha = 2^m(1-\alpha)$.
Then virtual network~\eqref{def:nn-virtual-bias} is $\alpha$-averaged if
\begin{equation}\label{e:condalphanew-bias}
\sqrt{b_\alpha} -\sqrt{a_{1,m}} \leq 2^m \alpha - 2\theta_m\;.
\end{equation}
\end{proposition}
\begin{proof}
Let us calculate
the operator norms of $ U $ and 
$ U - \gamma_{\alpha} \1 $,
where $U$ is given by \eqref{def:U-bias}.

\paragraph{Norm of $U$. }
Applying Lemma~\ref{e:normUni-bias} when $i=1$ and $n=m$ yields
\[
\| U\|^2 = a_{1,m} 
\; .
\]

\paragraph{Norm of $U - \gamma_{\alpha} \1$. }
We follow the same reasoning as in the proof of Lemma~\ref{e:normUni-bias} \textcolor{black}{(see also Appendix \ref{secA0})}.
We have
\[
\begin{split}
\| U -\gamma_{\alpha} \1 \|^2 
                        =  \underset{z=(\zp)_p, \|z\| =1}{\sup} 
\sum_p \zp^\top B_p\, \zp
\; ,
\end{split}
\]
where $B_p$ is the symmetric positive semidefinite matrix given by
\[B_p = \left(
\begin{array}{cc}
 (\beta_{1,m,p}-\gamma_{\alpha})^2    & \sqrt{\varpi_p}(\beta_{1,m,p} -\gamma_{\alpha})\widetilde{\beta}_{1,m,p}  \\
  \sqrt{\varpi_p}(\beta_{1,m,p} -\gamma_{\alpha})\widetilde{\beta}_{1,m,p}     &  (\eta_{1,m}-\gamma_{\alpha})^2 +\varpi_p\widetilde{\beta}_{1,m,p}^2
\end{array}
\right) \; .
\]
By definition of the spectral norm, 
\begin{equation}\label{e:norm2Ualpha-bias}
\| U -2^m(1-\alpha) \1 \|^2  = \sup_p\, \nu_{p}\; ,
\end{equation}
where, for every integer $p$, $\nu_{p}$ is  the maximum eigenvalue of $B_{p}$.
The two eigenvalues of this matrix are the roots of the polynomial 
\begin{align}
(\forall \nu \in \mathbb{R})\quad
    \operatorname{det} (B_{p} - \nu \1_2) 
                           = & \;
 \nu^2 - ( (\beta_{1,m,p}-\gamma_\alpha)^2  +  (\eta_{1,m}-\gamma_\alpha)^2 +\varpi_p\widetilde{\beta}_{1,m,p}^2 ) \ \nu 
\nonumber\\&+(\beta_{1,m,p}-\gamma_\alpha)^2 (\eta_{1,m}-\gamma_\alpha)^2
\; .
\end{align}
Solving the corresponding second-order equation leads to
\begin{multline}
\!\!\!\!\nu_p =  
%&
\frac{1}{2} 
\bigg(
 (\beta_{1,m,p}-\gamma_\alpha)^2  +  (\eta_{1,m}-\gamma_\alpha)^2 +\varpi_p\widetilde{\beta}_{1,m,p}^2 \nonumber\\
%& 
\; + \sqrt{\big( (\beta_{1,m,p}-\gamma_\alpha)^2  +  (\eta_{1,m}-\gamma_\alpha)^2 + \varpi_p\widetilde{\beta}_{1,m,p}^2 \big)^2 - 4  (\beta_{1,m,p}-\gamma_\alpha)^2(\eta_{1,m}-\gamma_\alpha)^2}
\bigg) \; .
\end{multline}
Then, it follows from~\eqref{e:norm2Ualpha-bias} that $\|U-\gamma_{\alpha}\1\|^2 = b_\alpha$.
\paragraph{Conclusion of the proof.} 
Based on the previous calculations, Condition \eqref{e:condalphanew-bias} is equivalent to \eqref{e:condalphaold}.
In addition, note that, \textcolor{black}{as shown in Appendix \ref{secA0}, for every $n\in \{1,\ldots,m\}$, $Q_n$ 
in \eqref{def:nn-virtual-bias} is firmly nonexpansive}. 
%since $R_n$ and $\1$ are.
By applying now Proposition~\ref{prop:oldres}\ref{prop:oldresii},
we deduce that, when Condition~\eqref{e:condalphanew-bias} holds, 
 virtual network~\eqref{def:nn-virtual-bias} is $\alpha$-averaged.
\end{proof}

\begin{remarque}
Condition \eqref{e:condalphanew-bias} just provides a sufficient condition for the averagedness of virtual network~\eqref{def:nn-virtual-bias}.
\end{remarque}
We conclude this subsection by a result emphasizing the interest of the leakage factors:
\begin{proposition}\label{prop:leakage}
Let $S$ be the virtual Model~\eqref{def:nn-virtual-bias} 
without leakage factor, i.e.,  for every $n \in \{ 1, \ldots, m \}$, $\eta_n =1$.
Assume that 
there exists $(x,b,b')\in \X^3$
such that \textcolor{black}{the output $x_m$ of the unfolded algorihm associated with 
$(x,b)$ is distinct from its output $x'_m$ associated with $(x,b')$.}
%such that $S(x,b) \neq S(x,b')$.
Then the Lipschitz constant of $S$ is greater than 1.
%Then, there does not exist $\alpha \in ]0,1[$ such that $S$ is $\alpha$-averaged. 
\end{proposition}

\begin{proof}
If $\eta_{1,m}=1$ then,
for every $n \in \{ 1, \ldots, m \}$,
\begin{equation}
    \label{def:nn-virtual}
z_n = \left( \begin{array}{c}
 x_n \\
 b_n
\end{array}
\right) ,
\quad
z_{n} = Q_n(U_n z_{n-1})
\; ,
\quad
\text{with}
\quad
\begin{cases}
\displaystyle Q_n = \left( \begin{array}{c}
  R_n   \\
  \1
\end{array} \right)
\; , \\
\\
\displaystyle U_n = \left( \begin{array}{cc}
   W_n  &  V_n \\
   0    & \1
\end{array}\right)
\; .
\end{cases}
\end{equation}
Suppose that the
virtual network $S$ is nonexpansive.
%$\alpha$-averaged
%with $\alpha \in ]0,1[$.
%This means that
%$R = (1-1/\alpha) \1 + S/\alpha $ is nonexpansive.
Thus, for every $(x,b)\in \X^2$ and $(x',b')\in \X^2$,  $ %\|R(x,b) - R(x',b')\|^2 
\|S(x,b)-S(x',b')\|
\leq \|x-x'\|^2 + \|b -b'\|_\varpi^2 $. 
\textcolor{black}{Let $x_m$ (resp. $x_m'$) denote the first
component in $\X$ of $S(x,b)$ (resp. $S(x',b')$).}
%Let $x_m = S(x,b)$ and $x'_m = S(x',b')$.
Since
\[
\begin{split}
%\|R(x,b) - R(x,b')\|^2 
\|S(x,b)-S(x,b')\|^2
  = & \; \Big\| 
  %(1-1/\alpha) \left( \begin{array}{c}
  %  x-x'  \\
  %    b-b'
%\end{array} \right)
 %+ 1/\alpha 
 \left( \begin{array}{c}
      \textcolor{black}{x_{m}-x'_{m}}  \\
      b-b'
\end{array} \right) \Big\|^2 \\
  = & \; \|
  %(1-1/\alpha) (x-x') + 1/\alpha (
  \textcolor{black}{x_{m}-x'_{m}}%) 
  \|^2 +\| b -b' \|_\varpi^2 \\
  \leq & \; \|x -x'\|^2 + \|b -b'\|_\varpi^2
  \; ,
\end{split}
\]
we deduce that
\[
 \|%(1-1/\alpha) (x-x') + 1/\alpha (
 \textcolor{black}{x_{m}-x'_{m}}%) 
 \|
 \leq 
 \|x -x'\|
 \; .
\]
Consequently, when $x = x'$, we have $x_m = x'_m$
whatever the choice of $b\neq b'$.
This contradicts our standing assumption.
%On the other, since $\lambda_m > 0$, if $R_m$ is not constant, then there exists $(b,b') \in \X^2$ such that
%$x_m = R_m(W_m x_{m-1}+\lambda_m b) \neq R_m(W_m x_{m-1}+\lambda_m b') = x'_m$.
\end{proof}

%Note that the assumption made 
%on $S$ is weak since it just means that the unfolded network is sensitive to the observed data. 
\textcolor{black}{Note that the assumption made on $S$ is relevant from the applicative viewpoint. It is typically satisfied if we have two different ground truth data 
$(\overline{x},\overline{x}')\in \X^2$ and we observe
$y = T \overline{x}+w$
and $y' = T\overline{x}'+w$
with $b=T^*y \neq T^*y'=b'$,
while the same initialization $x$ is used in the iterative
recovery process. Then, the condition $x_m\neq x'_m$ just means
that the unfolded network is sensitive to the observed data, as can be expected from a model
trained to perform suitable restoration.
}\\
The above result 
%extends to the limit case when $\alpha = 1$ showing 
shows
that the virtual network without leakage factors cannot be %nonexpansive.
averaged since any averaged operator is nonexpansive.

%%%%%%%%%%%%%%%%%%%%%%%%%%%%%%%%%%%%%%%%%%
\subsection{Link with the original neural network -- direct approach}\label{se:directlink}
In this subsection we go back to our initial model defined by~\eqref{def:modelNN-leakage}.
For simplicity, we assume that, for every 
$p$, $\varpi_p = 1$.
We consider two distinct inputs $z =(x,b)$ and $z'=(x',b')$ in $\X \times \X$. 
%The distance between $z$ and $z'$ is
%\[
%\|z'-z\| = \sqrt{\|x'-x\|^2 + \|b'-b\|^2}
%\; .
%\]
%
Let $z_m= (x_m,b_m)$ and $z_m'= (x'_m,b'_m)$
be the outputs of the $m$-th layer of virtual network~\eqref{def:nn-virtual-bias}
associated with inputs $z$ and $z'$, respectively.
Then, 
\[
\begin{split}
 \| z'_m - z_{m} \|^2 = & \; \| x'_{m} - x_{m}\|^2 + \| b'_{m} - b_{m}\|^2
\\
=& \; \| x'_{m} - x_{m}\|^2 + \eta_{1,m}^2 \| b' - b\|^2 
\; ,\\
\end{split}
\]
and, by applying Proposition~\ref{p:LipVNN-bias},
\[
\begin{split}
 \| z'_{m} - z_{m} \|^2 \leq \;
 \frac{\theta_m^2}{2^{2(m-1)}}\left(
 \| x'-x\|^2 + \| b'-b \|^2
 \right)
 \; .\\
\end{split}
\]
Then, the following inequality allows us to quantify the Lipschitz properties of the neural 
network~\eqref{def:modelNN} with an error on $b_0$:
\[
\| x'_{m} - x_{m} \|^2 \leq \;
 \frac{\theta_m^2}{2^{2(m-1)}}
 \| x'-x\|^2 
 + \left( \frac{\theta_m^2}{2^{2(m-1)}} - \eta_{1,m}^2 \right)\| b'-b \|^2
 \; .
\]
%\subsubsection{First estimates of the Lipschitz constant of the initial model}
Two particular cases are of interest:
\begin{itemize}
    \item If the network is initialized with a fixed signal, say $x_{0} = x'_{0}= 0$, then
 \[
 \| x'_{m} - x_{m} \|^2 \leq \; 
\left( \frac{\theta_m^2}{2^{2(m-1)}} - \eta_{1,m}^2 \right) 
\; \| b'-b\|^2
\; .
\]
So, a Lipschitz constant with respect to the 
input data $T^* y$ is 
\begin{equation}\label{e:Lipreal1-bias}
\vartheta_m = \sqrt{\frac{\theta_m^2}{2^{2(m-1)}} - \eta_{1,m}^2}.
\end{equation}
This shows, in particular, that the Lipschitz constant of the virtual network $\theta_m/2^{m-1}$ cannot be lower than $\eta_{1,m}$. This result is consistent with the observation made at the end of Section \ref{se:stabresVN}.
\item On the other hand, if the initialization is dependent
on the observed image, i.e. $x_{0}= b$ and $x'_{0} = b'$,
\[
\| x'_{m} - x_{m} \|^2 \leq \; 
\left(\frac{\theta_m^2}{2^{2m-3}} -\eta_{1,m}^2\right) 
\; \| b'-b\|^2
\; .
\]
So a higher Lipschitz constant value w.r.t. to the input data is obtained:
    \begin{equation}
\vartheta_m = \sqrt{\frac{\theta_m^2}{2^{2m-3}} -\eta_{1,m}^2}.
\end{equation}
\end{itemize}

\subsection{Use of a semi-norm}
Proposition \ref{prop:leakage} 
%and Remark \ref{re:directpasbon} 
suggests that we may need a finer strategy to evaluate 
the nonexpansiveness properties of
Model~\eqref{def:modelNN}.
%We consider Virtual Model~\eqref{def:nn-virtual}, where $\eta_n =1$
%for $n \in \{ 1, \ldots, m \} $.
On the product space $\X \times \X$, we define the semi-norm which takes only into account the first component of the vectors:
\begin{equation}
\label{def:seminormv}
    z = (x,b) \mapsto |z| = \|x\|.
\end{equation}
Let $L\colon \X\times \X \to \X\times \X$ be any bounded linear operator and, for every $z\in \X \times \X$, 
let $Lz = ((Lz)_{\rm x},(Lz)_{\rm b})$.
We define the associated operator semi-norm
\begin{equation}
    \label{def:seminorm}
| L | = \underset{\|z\| = 1 }{\text{sup}} \| (Lz)_{\rm x}\|
\; .
\end{equation}
The following preliminary result will be useful subsequently.
\begin{lemme}\label{e:normUni-proj}
Assume that $\X \times \X$ is equipped with the norm defined by \eqref{e:normweight}.
Let $m \in \mathbb{N}\setminus \{0\}$.
For every $n\in \{1,\ldots,m\}$ and 
$i\in \{1,\ldots,n\}$, 
the seminorm  $|U_n\circ \cdots \circ U_i|$ is 
equal to $\sqrt{\overline{a}_{i,n}}$ with
\begin{equation}\label{e:defain-proj}
\overline{a}_{i,n} = \sup_p
\left(
\beta_{i,n,p}^2 + \varpi_p \widetilde{\beta}_{i,n,p}^2
\right)
\le a_{i,n}
\; .
\end{equation}
\end{lemme}

\begin{proof}
The inequality
$\overline{a}_{i,n} \le
a_{i,n}$ follows from the fact that
$|U_n\circ \cdots \circ U_i|\le \|U_n\circ \cdots \circ U_i\|$.\\
The seminorm of $U_n\circ \cdots \circ U_i$ is the same as the norm
of $U_n\circ \cdots \circ U_i$ where $\eta_{n}$ has been set to 0.
The result thus follows from Lemma \ref{e:normUni-bias} where $\eta_{i,n}= 0$.\\
%Let us now show that, for every 
%$p$,
%\begin{equation}
%\beta_{i,n,p}^2 + \widetilde{\beta}_{i,n,p}^2
%\le \frac{1}{2}
%\left(
%\beta_{i,n,p}^2 + \widetilde{\beta}_{i,n,p}^2 + \eta_{i,n}^2+ %\sqrt{(\beta_{i,n,p}^2  + \widetilde{\beta}_{i,n,p}^2  +  \eta_{i,n}^2 )^2
%        - 4\beta_{i,n,p}^2 \eta_{i,n}^2}
%\right) \; ,
%\end{equation}
%that is
%\begin{equation}
%\label{e:ineqbaralesa}
%\beta_{i,n,p}^2 + \widetilde{\beta}_{i,n,p}^2
%- \eta_{i,n}^2
%\le  \sqrt{(\beta_{i,n,p}^2  + \widetilde{\beta}_{i,n,p}^2  +  \eta_{i,n}^2 )^2
%        - 4\beta_{i,n,p}^2 \eta_{i,n}^2} \; .
%\end{equation}
%The inequality obviously holds
%if $\beta_{i,n,p}^2 + \widetilde{\beta}_{i,n,p}^2
%\le \eta_{i,n}^2$. Let us
%focus on the case when
%$\beta_{i,n,p}^2 + \widetilde{\beta}_{i,n,p}^2
%> \eta_{i,n}^2$. Inequality
%\eqref{e:ineqbaralesa} is then equivalent to
%\begin{align}
%&(\beta_{i,n,p}^2 + \widetilde{\beta}_{i,n,p}^2
%- \eta_{i,n}^2)^2
%\le  (\beta_{i,n,p}^2  + \widetilde{\beta}_{i,n,p}^2  +  \eta_{i,n}^2 )^2
%        - 4\beta_{i,n,p}^2 \eta_{i,n}^2\nonumber\\
%\Leftrightarrow\quad
%& -2 (\widetilde{\beta}_{i,n,p}^2+\beta_{i,n,p}^2)
%\eta_{i,n}^2 \le
%2 (\widetilde{\beta}_{i,n,p}^2-
%\beta_{i,n,p}^2)\eta_{i,n}^2 \; ,
%\end{align}
%which is always satisfied.
%We deduce that 
\end{proof}

\subsubsection{Network with input $(x_0,b_0)$ and output $x_m$}
\label{se:x0b0xm}
The following result for the quantification of the Lipschitz regularity
is an offspring of Proposition~\ref{p:LipVNN-bias}.
\begin{proposition}
\label{p:LipVNN-proj}
Let $m \in \mathbb{N}\setminus \{0\}$. 
For every $i\in \{1,\ldots,n\}$ and $n\in \{1,\ldots,m-1\}$,
let $a_{i,n}$ be defined by
\eqref{e:defain-bias} and
let $\overline{a}_{i,m}$ be given by
\eqref{e:defain-proj}.
Set $\theta_0 = 1$ and
define 
\begin{align}
&(\forall n \in \{1,\ldots,m-1\})
\quad 
\theta_n = 
\sum_{i=1}^n \theta_{i-1} \sqrt{a_{i,n}}\;,\\
&\overline{\theta}_m = 
\sum_{i=1}^m \theta_{i-1} \sqrt{\overline{a}_{i,m}}\;.
\end{align}
Then
the network~\eqref{def:modelNN-leakage} with input $(x_0,b_0)$ and output $x_m$ is $\overline{\theta}_m/2^{m-1}$-Lipschitz. In addition,
\begin{equation}
    \sqrt{\overline{a}_{1,m}}
    \le \frac{\overline{\theta}_m}{2^{m-1}}
    \le \sqrt{\overline{a}_{m,m}}
    \Big(\prod_{n=1}^{m-1} a_{n,n}\Big)^{1/2}.
\end{equation}

\end{proposition}

\begin{proof}
Network~\eqref{def:modelNN-leakage}
can be expressed as $R_m \circ \overline{U}_m \circ Q_{m-1}\circ U_{m-1}\circ \cdots \circ Q_1 \circ U_1$ where 
\begin{equation}\label{e:oUm}
\overline{U}_m = D_{\rm x}\circ U_m
\end{equation}
and $D_{\rm x}$ is the decimation operator
\begin{equation}
    D_{\rm x}= [\1\quad 0].
\end{equation}
This  network has the same Lipschitz properties as the network in~\eqref{def:nn-virtual-bias} with $\eta_{m}= 0$.
The result can thus be deduced from Proposition \ref{p:LipVNN-bias} by setting $\eta_{1,m}= 0$.
\end{proof}

Note that the resulting estimate of the Lipschitz constant is smaller than the one obtained for the VNN. Indeed, according to 
Lemma \ref{e:normUni-proj},
\begin{equation}\label{e:overthetabettertheta}
\overline{\theta}_m = 
\sum_{i=1}^m \theta_{i-1} \sqrt{\overline{a}_{i,m}}
\le \sum_{i=1}^m \theta_{i-1} \sqrt{a_{i,m}} = \theta_m\;.
\end{equation}

\subsubsection{Network with a single input and a single output}\label{se:netsingisingo}
Another possibility for investigating Lipschitz  properties,
including averagedness,
consists of defining a network from $\X$ to $\X$. 
We will focus on specific networks of the form
\begin{equation}\label{e:NNreal}
R_m\circ \overline{U}_m \circ Q_{m-1}\circ U_{m-1}\cdots\circ
Q_1 \circ \widehat{U}_{1},
\end{equation}
where $\overline{U}_{m}$ is given by \eqref{e:oUm}
and
\begin{equation}
\widehat{U}_{1} = U_1 \begin{bmatrix}
F\\
\1
\end{bmatrix}\,.
\end{equation}
Hereabove, $F$ is a bounded linear operator from $\X$ to $\X$.
We will make the following additional assumption:\\
\fbox{
Operator $F$
can be diagonalized
in basis
$(v_{p})_{p}$, 
its eigenvalues being denoted
by $(\phi_p)_p$.
%in the same orthonormal set of eigenvectors $(v_{p})_{p}$.
}

This assumption encompasses the three following cases:
\begin{enumerate}
\item\label{caseNN1} The first one assumes that $x_0 = 0$ in \eqref{def:modelNN}, \textcolor{black}{by setting} 
$F = 0$.
%It is thus given by 
%\begin{equation}
%\widehat{U}_{1} = U_1 %\begin{bmatrix}
%0\\
%\1
%\end{bmatrix}\,.
%\end{equation}
\item\label{caseNN2} The second one assumes that $x_0=b_0$ in \eqref{def:modelNN}, \textcolor{black}{by setting}  $F = \1$ and, for every $p$, $\phi_p = 1$.
%It is thus given by
%\begin{equation}
%\widehat{U}_{1} = U_1 %\begin{bmatrix}
%\1\\
%\1
%\end{bmatrix}\,.
%\end{equation}
\item The third one assumes that
$x_0 = F b_0 = FT^* y$ where 
$F$ is a pre-filter, while 
$T$ and $(D_n)_{1\le n \le m}$ also are filtering operations. For instance,
$F$ could perform a rough 
restoration of the sought signal from the observed one $b_0$ \cite{Terris2019}, possibly by ensuring that $F T^*$ is a Wiener filter. In this case, $(\phi_p)_p$ corresponds to the complex-valued frequency response of this pre-filter at discrete frequencies indexed by $p$.
\end{enumerate}

Similarly to Lemma \ref{e:normUni-bias}, we obtain the following 
result. 
\begin{lemme}\label{e:normUni-proj2}
Assume that $\X \times \X$ is equipped with the norm defined by \eqref{e:normweight} and the input 
space is $(\X,\|\cdot\|_\varpi)$.
Let $m \in \mathbb{N}\setminus \{0,1\}$.
For every $n\in \{1,\ldots,m-1\}$,
the norm
of $U_n\circ \cdots \circ U_{2}\circ \widehat{U}_{1}$ is 
equal to $\sqrt{\widehat{a}_{1,n}}$ with
\begin{equation}\label{e:defain-bias2}
\widehat{a}_{1,n} = 
\sup_p \left(\varpi_p|\beta_{1,n,p}\phi_p+\widetilde{\beta}_{1,n,p}|^2\right) + \eta_{1,n}^2 
\end{equation}
%\begin{equation}\label{e:defain-bias2}
%\widehat{a}_{1,n} = 
%\begin{cases}
%\sup_p \widetilde{\beta}_{1,n,p}^2 + \eta_{1,n}^2  & \mbox{in %case \ref{caseNN1}}\\
%\sup_p \left((\beta_{1,n,p}+\widetilde{\beta}_{1,n,p})^2\right) + %\eta_{1,n}^2  & \mbox{in case \ref{caseNN2}}
%\end{cases}
%\end{equation}
and the norm of $\overline{U}_m\circ U_{m-1}\circ \cdots \circ U_{2}\circ \widehat{U}_{1}$
is  equal to $\sqrt{\widehat{a}_{1,m}}$ with
\begin{equation}\label{e:defain-projm2}
\widehat{a}_{1,m} = 
%\begin{cases}
%\sup_p
%\widetilde{\beta}_{1,m,p}^2 & \mbox{in case \ref{caseNN1}}\\
\sup_p \varpi_p
|\beta_{1,m,p}\phi_p+ \widetilde{\beta}_{1,m,p}|^{2} \;.
%& \mbox{in case \ref{caseNN2}.}
%\end{cases}
\end{equation}
%\begin{equation}\label{e:defain-projm2}
%\widehat{a}_{1,m} = 
%\begin{cases}
%\sup_p
%\widetilde{\beta}_{1,m,p}^2 & \mbox{in case \ref{caseNN1}}\\
%\sup_p
%(\beta_{1,m,p}+ \widetilde{\beta}_{1,m,p})^{2}  & %\mbox{in case \ref{caseNN2}.}
%\end{cases}
%\end{equation}
\end{lemme}

\begin{proof}
For every $n\in \{1,\ldots,m-1\}$,
\begin{equation}\label{e:oneonepreLip}
\begin{split}
\|U_n\circ \cdots \circ U_2
\circ \widehat{U}_1
b_0 \|^2 = & \; \| W_{1,n} F b_0 + \widetilde{W}_{1,n} b_0\|^2 + \eta_{1,n}^2 \|b_0\|_\varpi^2\\
= & \sum_p \big(|\beta_{1,n,p}\phi_p
+\widetilde{\beta}_{1,n,p}|^2
+\varpi_p^{-1}\eta_{1,n}^2\big) |\zeta_p|^2\\
\le & \sup_p \big(\varpi_p|\beta_{1,n,p}\phi_p
+\widetilde{\beta}_{1,n,p}|^2
+\eta_{1,n}^2\big)
\|b_0\|_\varpi^2
\;,
\end{split}
\end{equation}
where $b_0$ has been decomposed as
\[
b_0 = \sum_p \zeta_p v_p.
\]
It can be deduced that
\[
\|U_n\circ \cdots \circ U_2
\circ \widehat{U}_1\|^2 = 
\widehat{a}_{1,n}\;.
%\sup_p \big(|\beta_{1,n,p}\phi_p
%+\widetilde{\beta}_{1,n,p}|^2
%+\eta_{1,n}^2\big).
\]
Equation \eqref{e:oneonepreLip} remains valid when 
$n=m$, by setting $\eta_m = 0$, i.e. 
$\eta_{1,m} = 0$. This leads to the expression
of $\widehat{a}_{1,m}$.
\end{proof}

By proceeding similarly to the proof of Proposition \ref{prop:ab-bias}, we deduce the following 
stability result.
\begin{proposition}
\label{p:LipVNN-proj2}
Let $m \in \mathbb{N}\setminus \{0,1\}$. 
For every $i\in \{2,\ldots,n\}$ and $n\in \{1,\ldots,m-1\}$,
let $a_{i,n}$ be defined by
\eqref{e:defain-bias} and
let $\overline{a}_{i,m}$ be given by
\eqref{e:defain-proj}.
For every $n\in \{1,\ldots,m\}$, let $\widehat{a}_{1,n}$
be defined by \eqref{e:defain-bias2} and  \eqref{e:defain-projm2}.
Define $(\widehat{\theta}_n)_{1\le n \le m}$ recursively by
\begin{align}
\label{e:widehathetan}
&(\forall n \in \{1,\ldots,m-1\})
\quad 
\widehat{\theta}_n = \sqrt{\widehat{a}_{1,n}}+
\sum_{i=2}^n \widehat{\theta}_{i-1} \sqrt{a_{i,n}}\;,\\
\label{e:widehathetam}
&\widehat{\theta}_m = \sqrt{\widehat{a}_{1,m}}+
\sum_{i=2}^m \widehat{\theta}_{i-1} \sqrt{\overline{a}_{i,m}}\;,
\quad \widehat{\theta}_0 = 1\;.
\end{align}
Then
network \eqref{e:NNreal} is $\widehat{\theta}_m /2^{m-1}$-Lipschitz. In addition,
\begin{equation}
\sqrt{\widehat{a}_{1,m}} \le 
\frac{\widehat{\theta}_m}{2^{m-1}}
\le \left(\overline{a}_{m,m}
\Big(\prod_{n=2}^{m-1} a_{n,n}\Big) \widehat{a}_{1,1}\right)^{1/2}.
\end{equation}
\end{proposition}

As shown hereafter, the obtained Lipschitz constant is an increasing 
function of the pre-filter components 
$(\phi_p)_p$ when those are nonnegative and real-valued.
\begin{proposition}\label{prop:influencephi}
Assume that
\begin{equation}\label{e:limistep}
(\forall  n\in \{1,\ldots,m\})\quad 
\sup_p \lambda_n (\beta_{T,p}+\beta_{D_n,p}) \le 1. 
\end{equation}
Let us denote by $\widehat{\theta}_m(\phi)$
the Lipschitz parameter defined in the previous proposition for a given value of $\phi = (\phi_p)_p$. If $\phi = (\phi_p)_p$
and $\phi' = (\phi'_p)$ are two sequences of nonnegative reals, then
\[
(\forall p)\;\;\phi_p \ge \phi'_p
\quad 
\Rightarrow\quad \widehat{\theta}_m(\phi) \ge
 \widehat{\theta}_m(\phi').
\]
\end{proposition}

\begin{proof}
The real values of $(a_{i,n})_{1\le i \le n,1\le n \le m-1}$, $(\overline{a}_{i,m})_{1\le i \le n}$,
and $(\widehat{a}_{1,n})_{1\le n \le m}$,
defined respectively by \eqref{e:defain-bias},
\eqref{e:defain-proj}, and \eqref{e:defain-bias2}-\eqref{e:defain-projm2},
are nonnegative.
If Condition \eqref{e:limistep} holds, then
the constants defined in  \eqref{def:vp0},
\eqref{def:vp1}, and \eqref{def:vp-bias}
are nonnegative. Thus, according to \eqref{e:defain-bias2}-\eqref{e:defain-projm2}, for every
$n\in \{1,\ldots,m\}$, $\widehat{a}_{1,n}$
is an increasing function of the components of the associated pre-filter $\phi$.
It follows from \eqref{e:widehathetan} that
the monotonicity property holds
for $(\widehat{\theta_n})_{1\le n\le m-1}$
and, from \eqref{e:widehathetam},
that it extends to $\widehat{\theta}_m$.
\end{proof}
We deduce, in particular, that the Lipschitz constant is lower when 
initializing with  $x_0 = 0$ (case \textit{i)}) rather than 
$x_0 = b_0$ (case \textit{ii)}).\\

 \section{Numerical illustrations}
\label{section:numerics}
%%%%%%%%%%%%%%%%%%%%%%%%%%%%%%%%%%%%%%%%%%%%%%%%%%%%%%%%%%%%%%%%%%%%%%%%%%%%%%%%%%%%%%%%%%%%%%%%%%%%%%%%%%%%%%%%%%%%%%%%%%%%%%%%%

%\end{document}
\subsection{Stationary case}
We first consider the case when the parameters do not vary across the layers.
In this stationary case, we have:
for every $n\in \{1,\ldots,m\}$, 
$D_n = D$, $\lambda_n = \lambda$, 
$\eta_n = \eta$, and $\chi_n = \chi$. Details concerning this scenario are provided in Appendix \ref{secA1}.\\
We set $T$ as a two-dimensional $3\times 3$ \textcolor{black}{discrete} uniform blur such that $\|T\| = 1$.
In addition,
\[
D = \sqrt{\tau} \begin{bmatrix}
\nabla_{\rm H}\\
\nabla_{\rm V}
\end{bmatrix}
\]
where $\nabla_{\rm H}$ (resp. $\nabla_{\rm V}$) is the horizontal
(resp. vertical) discrete gradient operator and $\tau = 10^{-2}$.
Since $\X = \R^{N_1\times N_2}$ ($N_1 = N_2= 256$), the basis $(v_p)_{1\le p \le N_1N_2}$ is associated with the 2D discrete Fourier transform.
In our tests, we employ the weighted norm defined by
\begin{equation}\label{e:defnormwex}
(\forall p \in \{1,\ldots,N_1N_2\})\quad
\varpi_p = \beta_{T,p}^2+\epsilon
\end{equation}
with $\epsilon = 10^{-2}$. In the following experiments, the strong convexity modulus $\overline{\chi}$ is   zero.

Figure \ref{fig:thetam} illustrates the computed values of $\theta_m/2^{m-1}$ for the considered VNN, based on Proposition~\ref{p:LipVNN-bias}, as functions of the stepsize $\lambda$ and the leakage factor $\eta$, for one layer ($m=1$) and fifteen layers ($m=15$). The left part highlights the regions where the Lipschitz constant is less than one, hence the VNN is averaged. However, we observe that achieving nonexpansiveness of the virtual network requires relatively restrictive parameter choices.
Interestingly, small values of the leakage factor exhibit a stabilizing effect, as suggested by Proposition~\ref{prop:leakage}. Moreover, the region of nonexpansiveness is larger for $\theta_{15}/2^{14}$ compared to $\theta_1$, indicating that a multi-layer stability analysis yields more accurate results. This observation is further confirmed in Figure~\ref{fig:thetam-1}, which shows the difference between the separable estimate after 15 layers, given by $(\theta_1)^{15}$, and the computed value $\theta_{15}/2^{14}$.

\begin{figure}
\centering
\begin{tabular}{cc}
\includegraphics[height=4.5cm]{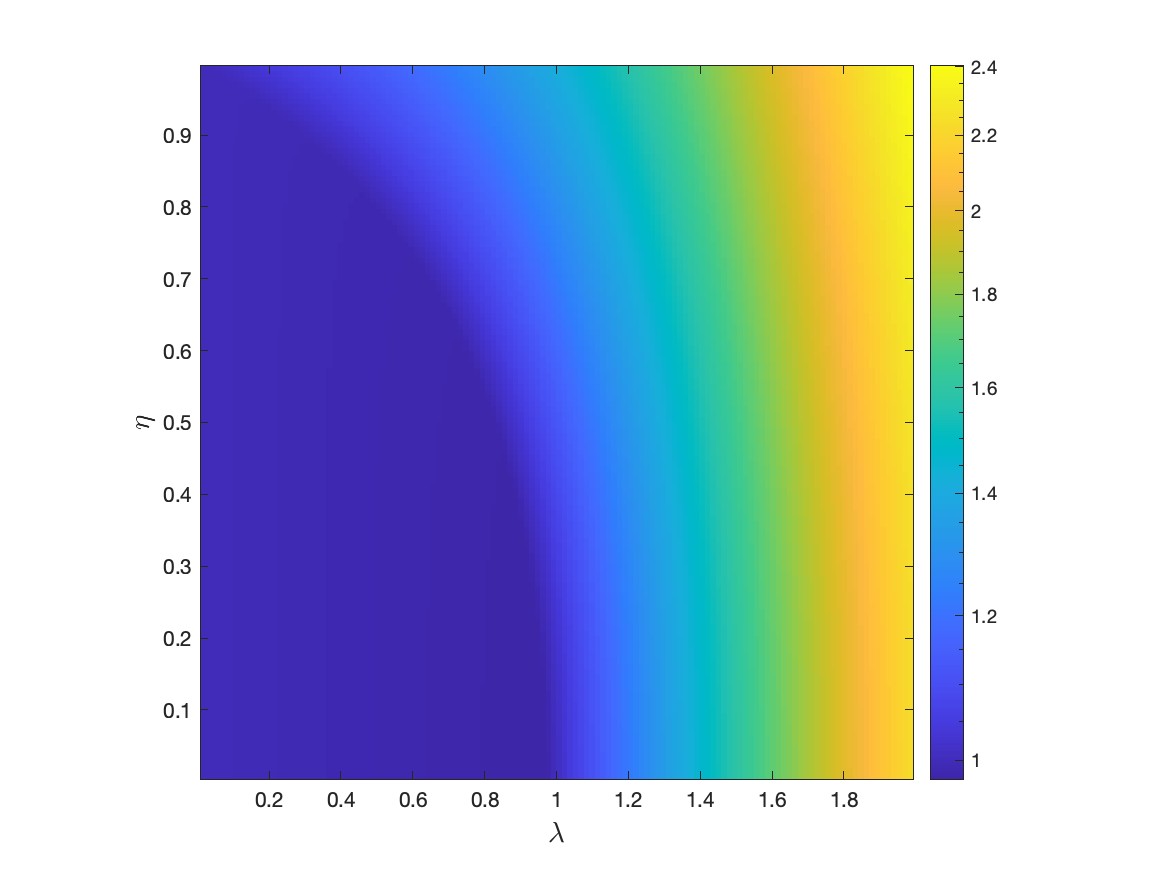}
&
\includegraphics[height=4.5cm]{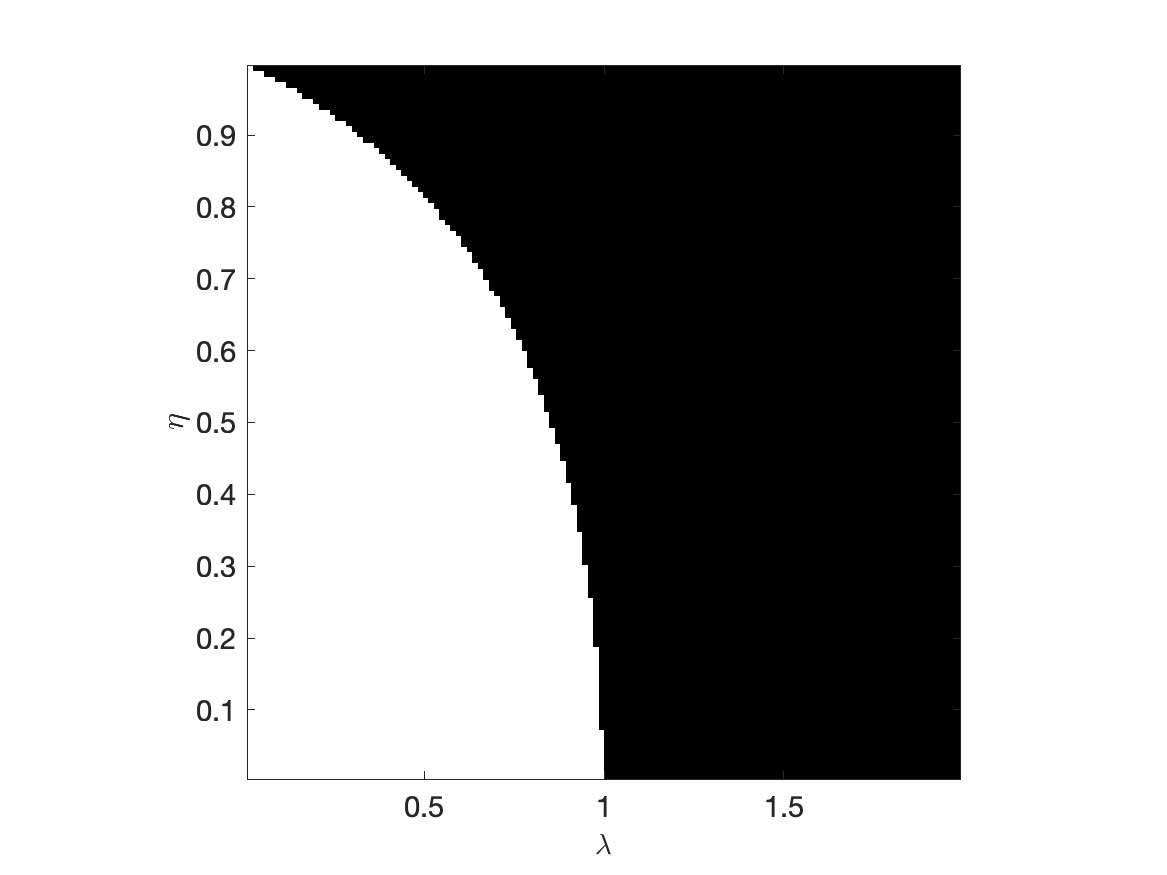}
\\
\includegraphics[height=4.5cm]{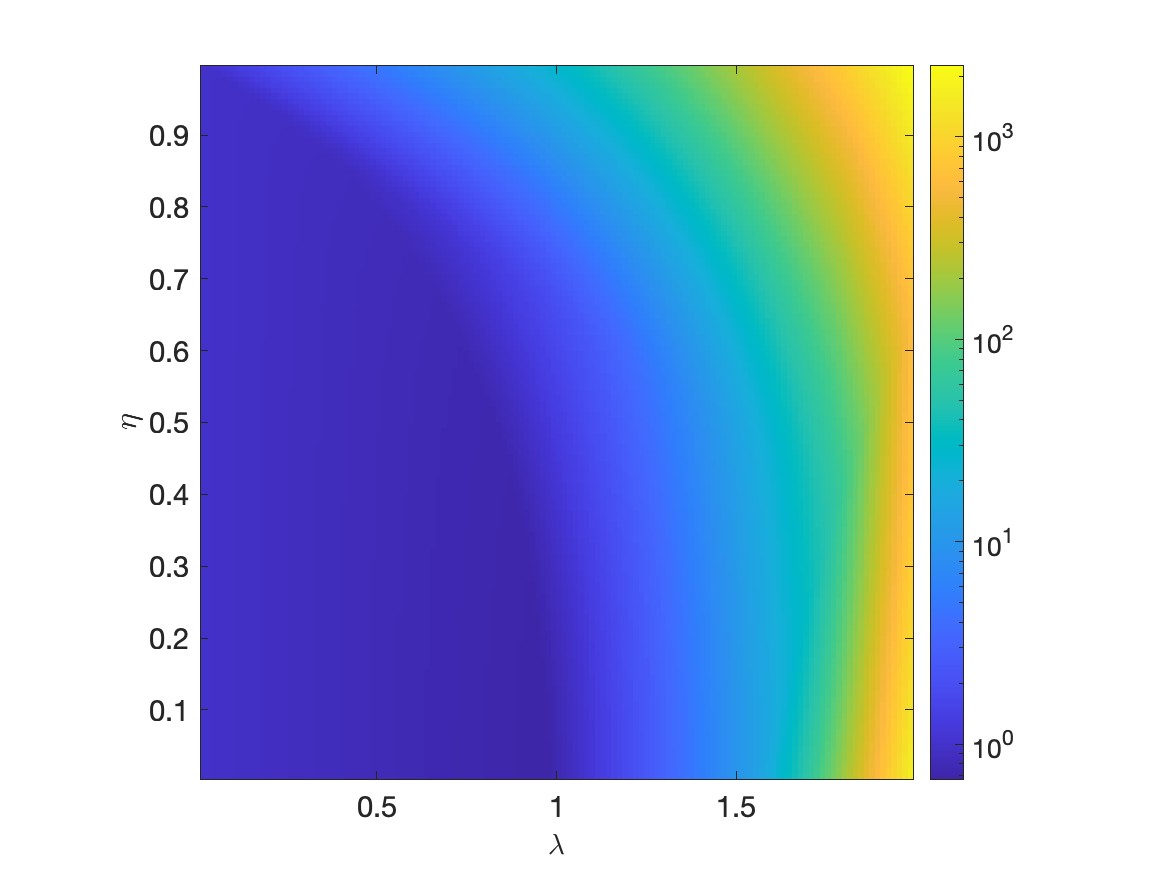}
&
\includegraphics[height=4.5cm]{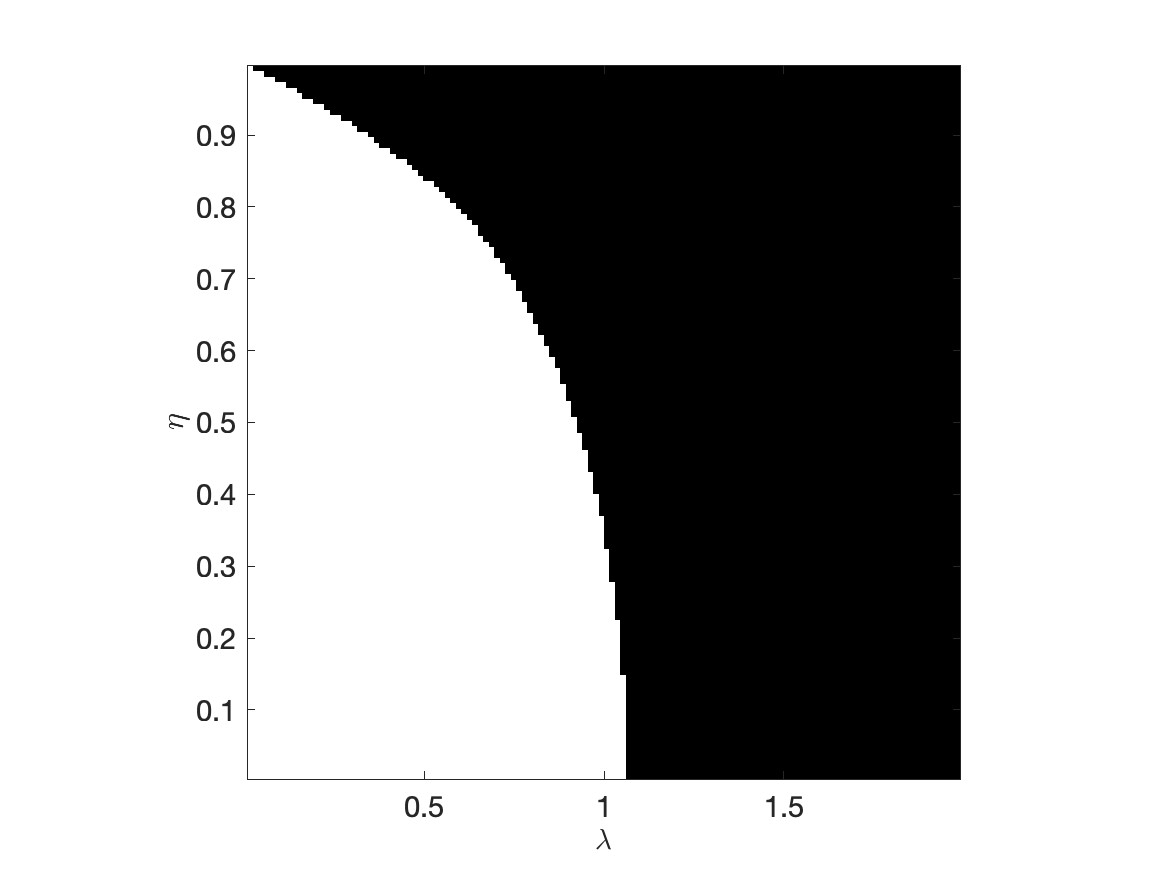}
\\
%$\scriptsize \widehat{\theta}_{15}$ & $\scriptsize [\widehat{\theta}_{15} \le 1]$
\end{tabular}
\caption{Stationary case: Lipschitz constant of the VNN as a function of $\lambda$ and $\eta$. First row: $\theta_1$, second one: $\theta_{15}/2^{14}$. First column: numerical value, second one: white if $\theta_m/2^{m-1} \le 1$.}\label{fig:thetam}
%\end{center}
\end{figure}

\begin{figure}
\centering
\includegraphics[height=4.5cm]{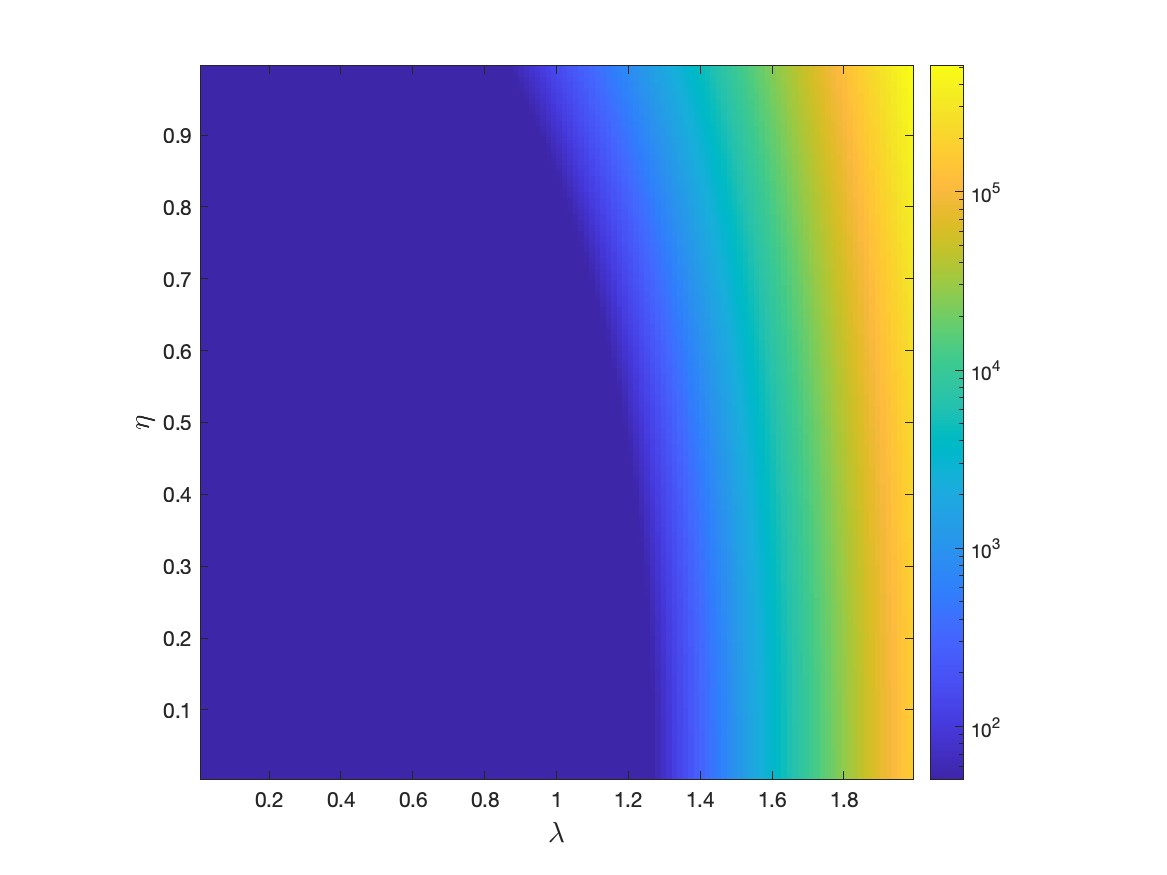}
\caption{Stationary case: Difference
$(\theta_1)^{15}-\theta_{15}/2^{14}$
for the VNN in the stationary case as a function of $\lambda$ and $\eta$.}\label{fig:thetam-1}
\end{figure}

Figure \ref{fig:thetab} illustrates a similar behavior for the Lipschitz constant $\overline{\theta}_{15}/2^{14}$, based on the use of a semi-norm on the VNN output, which has been computed by using Proposition~\ref{p:LipVNN-proj}. The bottom map displays the difference between $\theta_{15}/2^{14}$ and $\overline{\theta}_{15}/2^{14}$, which, as shown in \eqref{e:overthetabettertheta}, is always nonnegative.

Lastly, for the network with a single input and a single output, Figure \ref{fig:thetac} presents the Lipschitz constant $\widehat{\theta}_{15}/2^{14}$, computed using Proposition~\ref{p:LipVNN-proj2}, for two different operator choices, $F=0$ and $F=\1$. We observe that selecting $F=0$ results in greater stability compared to $F=\1$, which aligns with Proposition~\ref{prop:influencephi}.

\begin{figure}
\centering
\includegraphics[height=4.5cm]{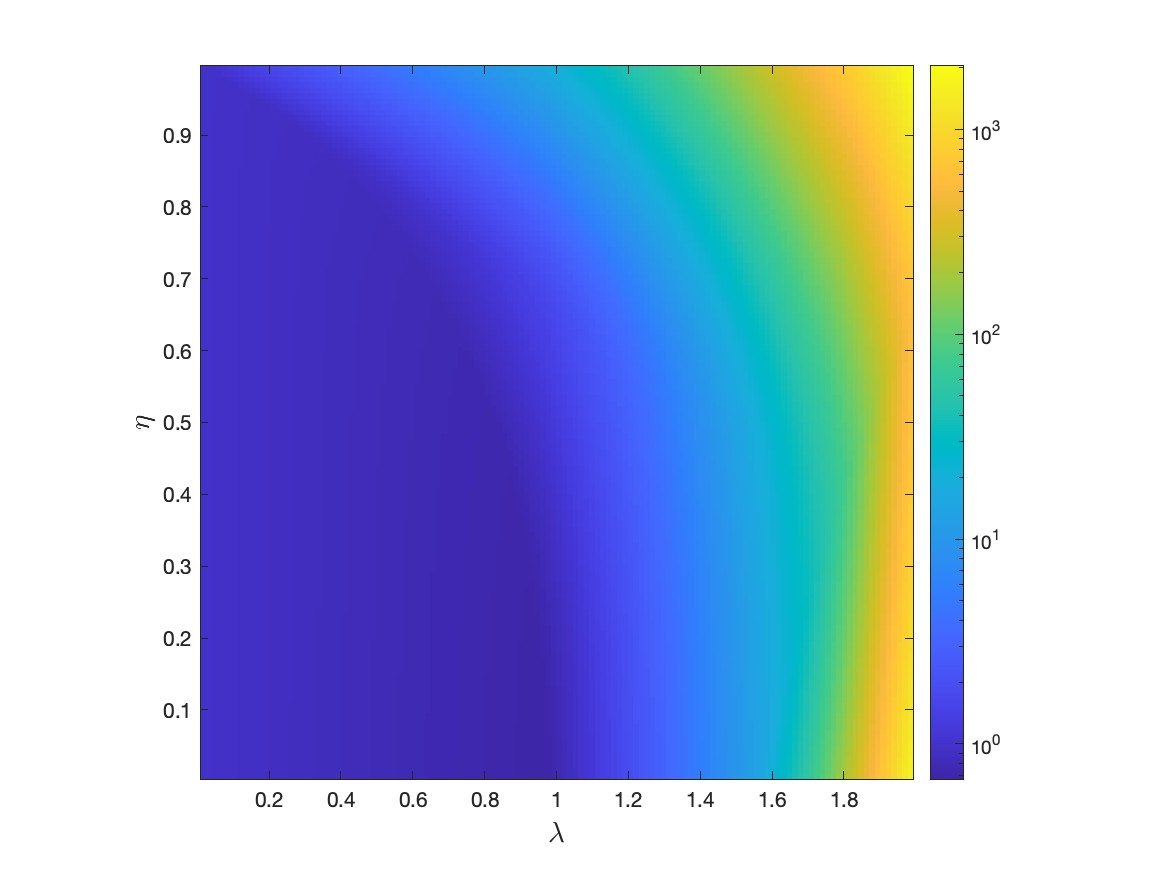}
\includegraphics[height=4.5cm]{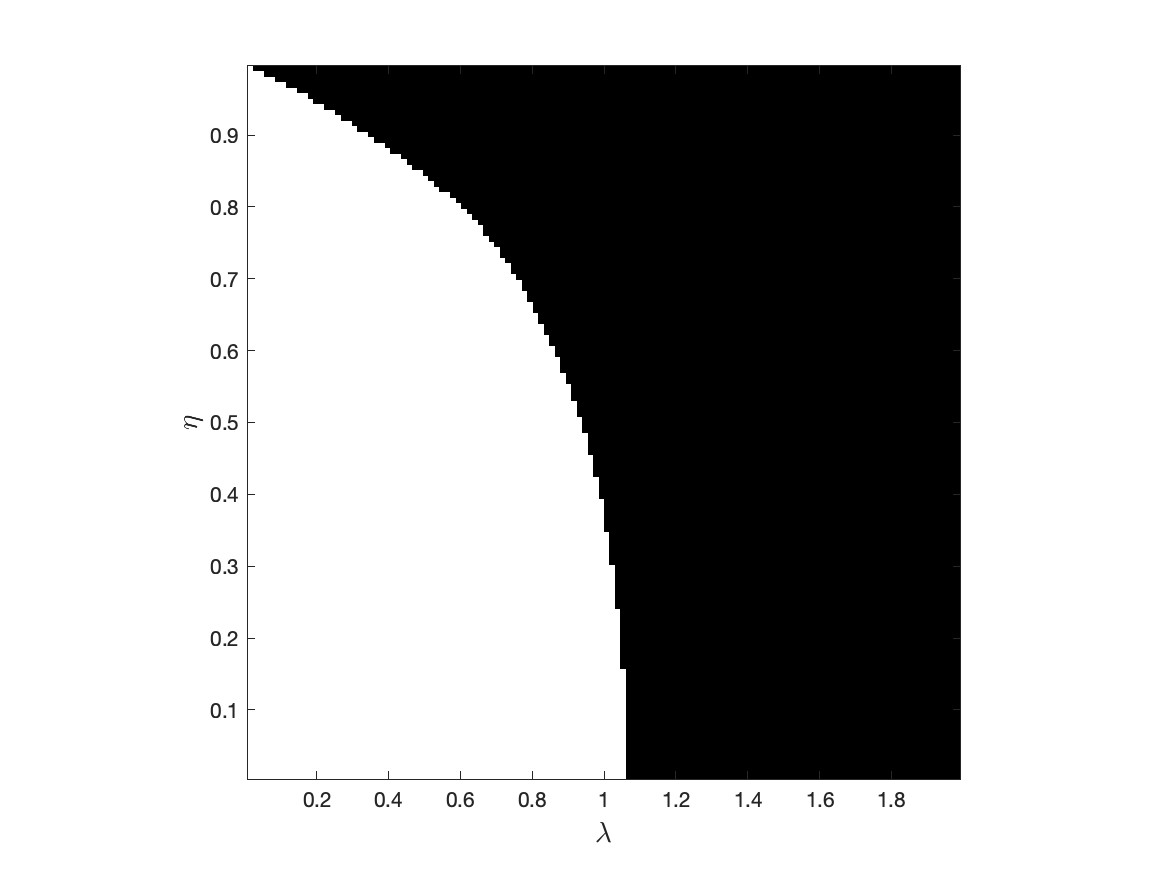}
\includegraphics[height=4.5cm]{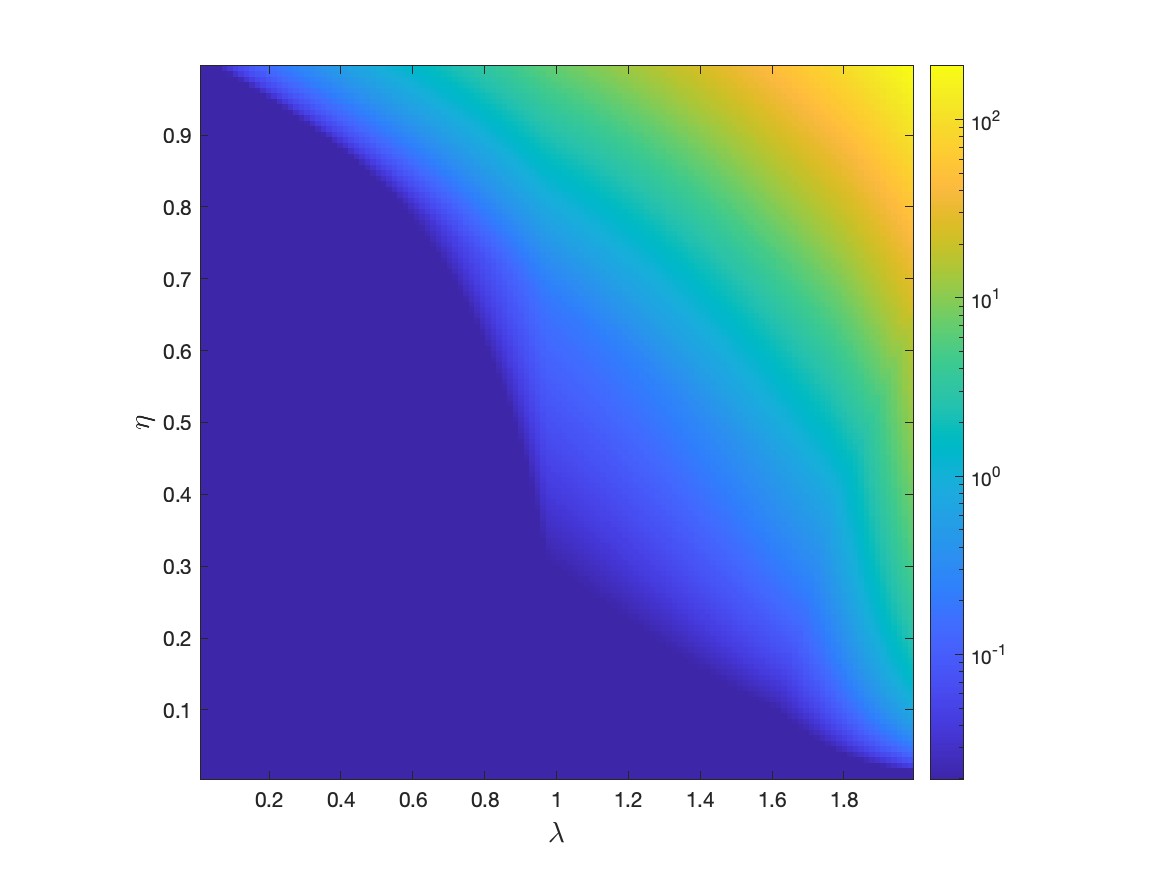}
\caption{Stationary case: Lipschitz constant of the network with inputs
$(x_0,b_0)$ and output $x_{15}$
as a function of $\lambda$ and $\eta$. 
Top left: $\overline{\theta}_{15}/2^{14}$.
Top right: white when $\overline{\theta}_{15}/2^{14}<1$.
\textcolor{black}{Bottom:} Difference
$\theta_{15}-\overline{\theta}_{15}$.}\label{fig:thetab}
\end{figure}

\begin{figure}
\centering
\begin{tabular}{cc}
\includegraphics[height=4.5cm]{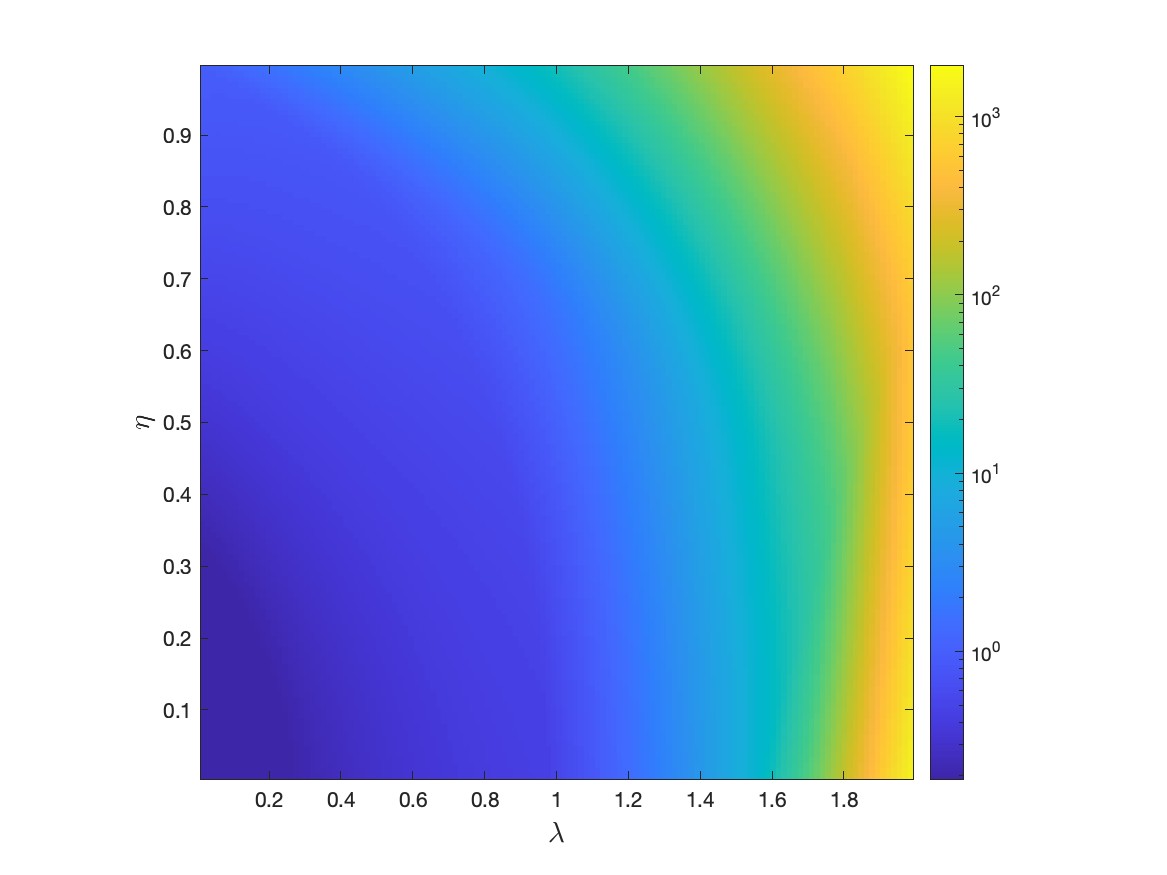}
&
\includegraphics[height=4.5cm]{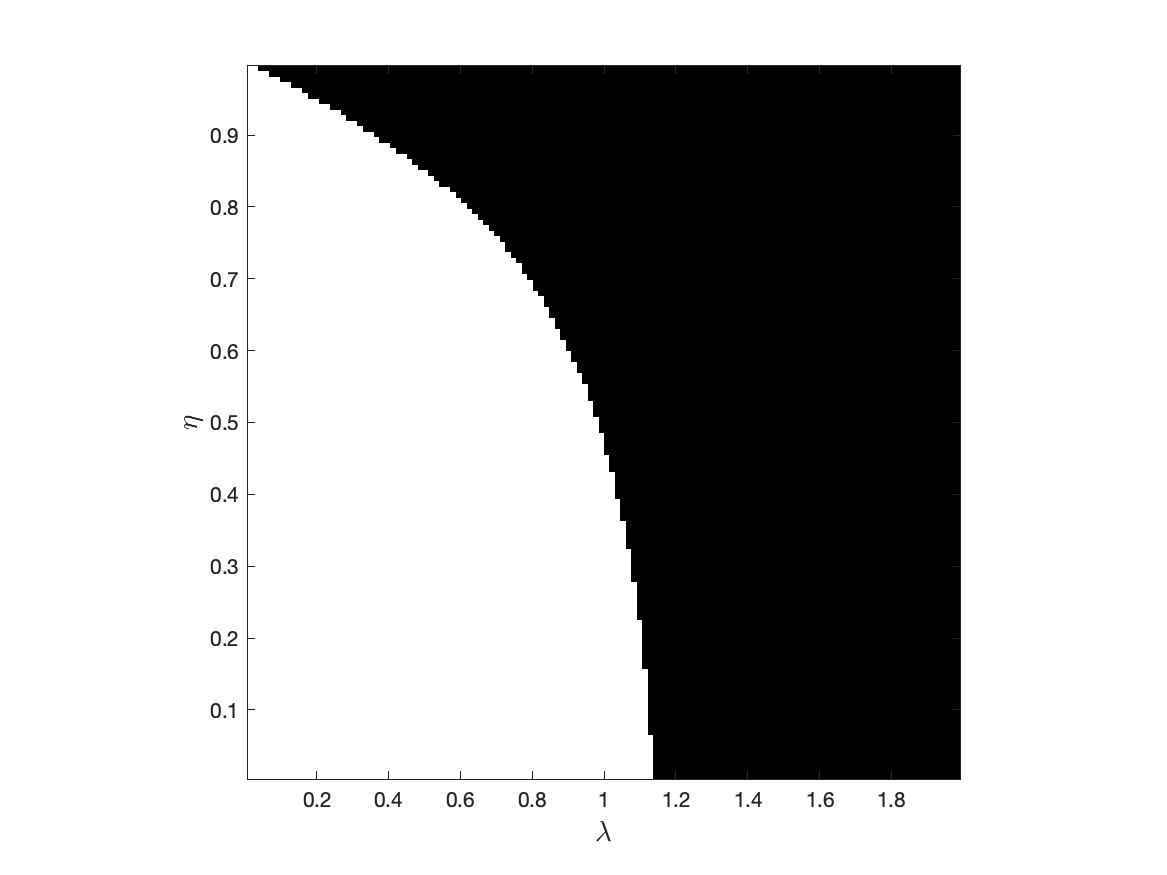}
\\
\includegraphics[height=4.5cm]{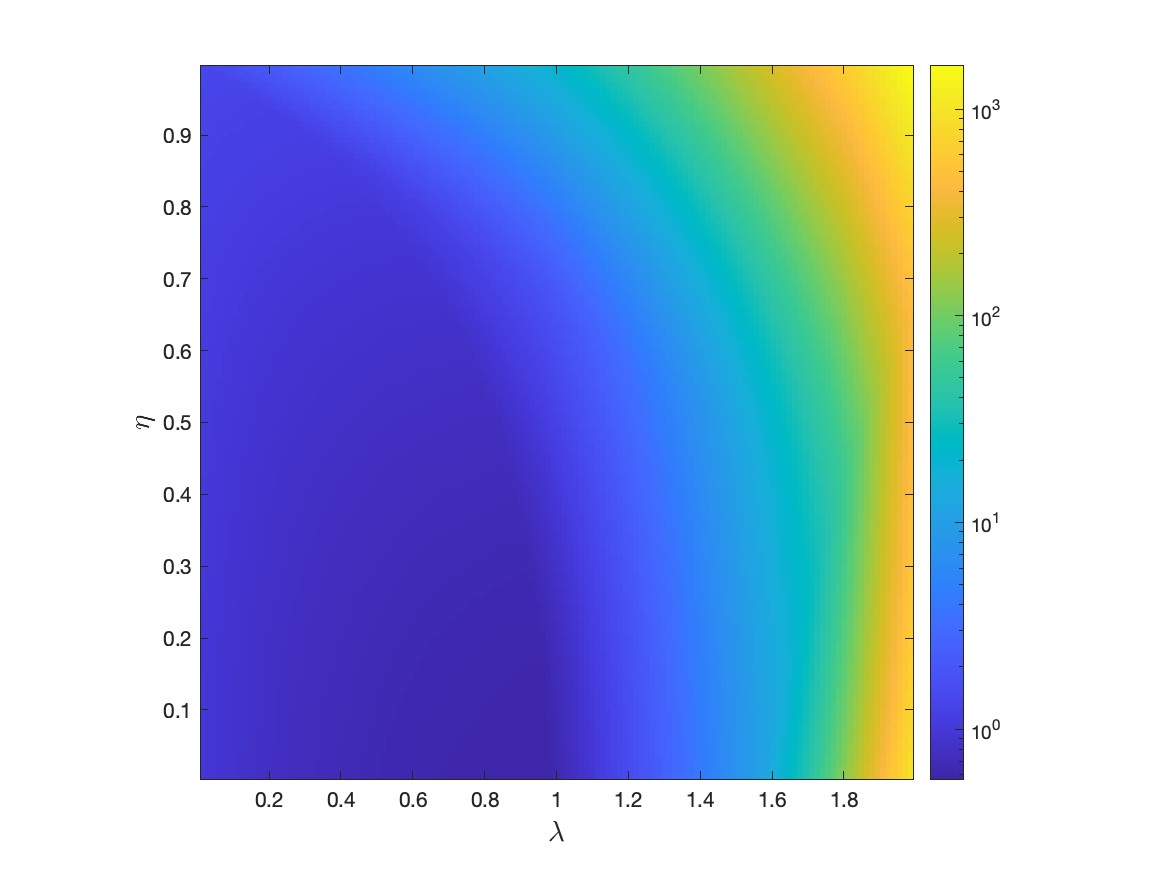}
&
\includegraphics[height=4.5cm]{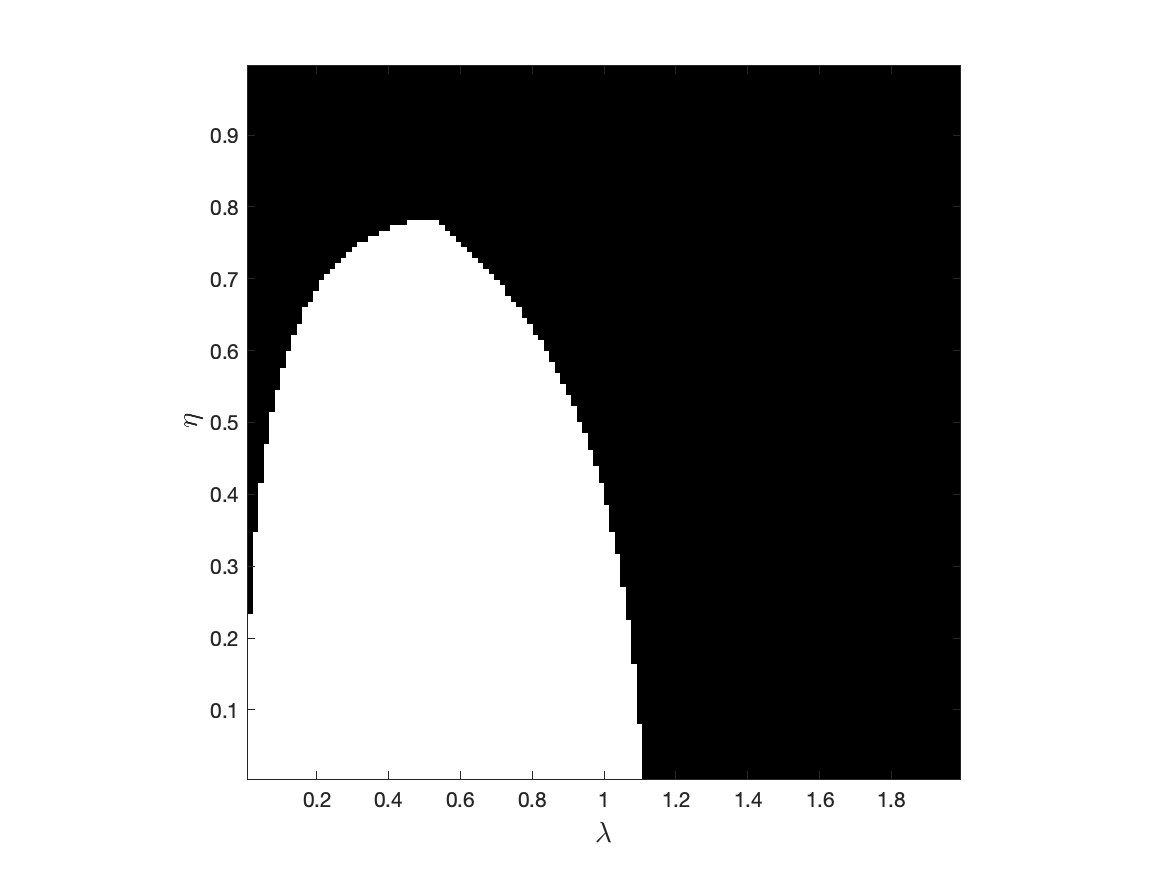}
\\
%$\scriptsize \widehat{\theta}_{15}$ & $\scriptsize [\widehat{\theta}_{15} \le 1]$
\end{tabular}
\caption{Stationary case: Lipschitz constant $\widehat{\theta}_{15}/2^{14}$ \textcolor{black}{as} a function of $\lambda$ and $\eta$. First row: $F = 0$, second one: $F = \1$. First column: numerical value, second one: white if $\widehat{\theta}_{15}/2^{14} \le 1$.}\label{fig:thetac}
%\end{center}
\end{figure}

\subsection{Nonstationary case}
We now consider the scenario most commonly encountered in unrolled algorithms, where the parameter values vary across layers. In our experiments, we set
\[
(\forall n \in \{1,\ldots,m\})\quad 
D_n = \sqrt{\tau_n} \begin{bmatrix}
\nabla_{\rm H}\\
\nabla_{\rm V}
\end{bmatrix},
\]
and choose $T$ as a two-dimensional \textcolor{black}{discrete} $5\times 5$ uniform blur such that $\|T\| = 1$.
Parameters $\{\tau_n\}_{1\le n \le m}
\subset [0.099,0.0249]$
and $\{\lambda_n\}_{1\le n \le m}
\subset [0.5128,0.9585]$, typical of those learned from an image dataset, are used.  \textcolor{black}{Constant 
values for $(\eta_n)_{1\le n \le m}$ and
$(\chi_n)_{1\le n \le m}$ have been chosen.}
They are set to
0.98 and $10^{-3}$, respectively, although setting the latter strong convexity moduli to zero results in only a slight reduction in the estimated Lipschitz regularity.
The weighted norm is defined similarly to the previous example (see
\eqref{e:defnormwex}).

Figure \ref{fig:nonstat} illustrates the variations in the Lipschitz constant estimates \(\theta_n/2^{n-1}\), \(\overline{\theta}_n/2^{n-1}\), and \(\widehat{\theta}_n/2^{n-1}\) as a function of the layer index \(n\). We also compare the proposed estimates with the lower and upper bounds derived in Propositions~\ref{p:LipVNN-bias}, \ref{p:LipVNN-proj}, and \ref{p:LipVNN-proj2}. It can be observed that the proposed bounds are significantly tighter than the separable lower bounds. While the resulting values exceed 1, they remain within a satisfactory order of magnitude when compared to those encountered in sensitive networks \cite{Neacsu2024}.  

\begin{figure}
\centering
\begin{tabular}{c}
\includegraphics[width=9cm,height=4.5cm]{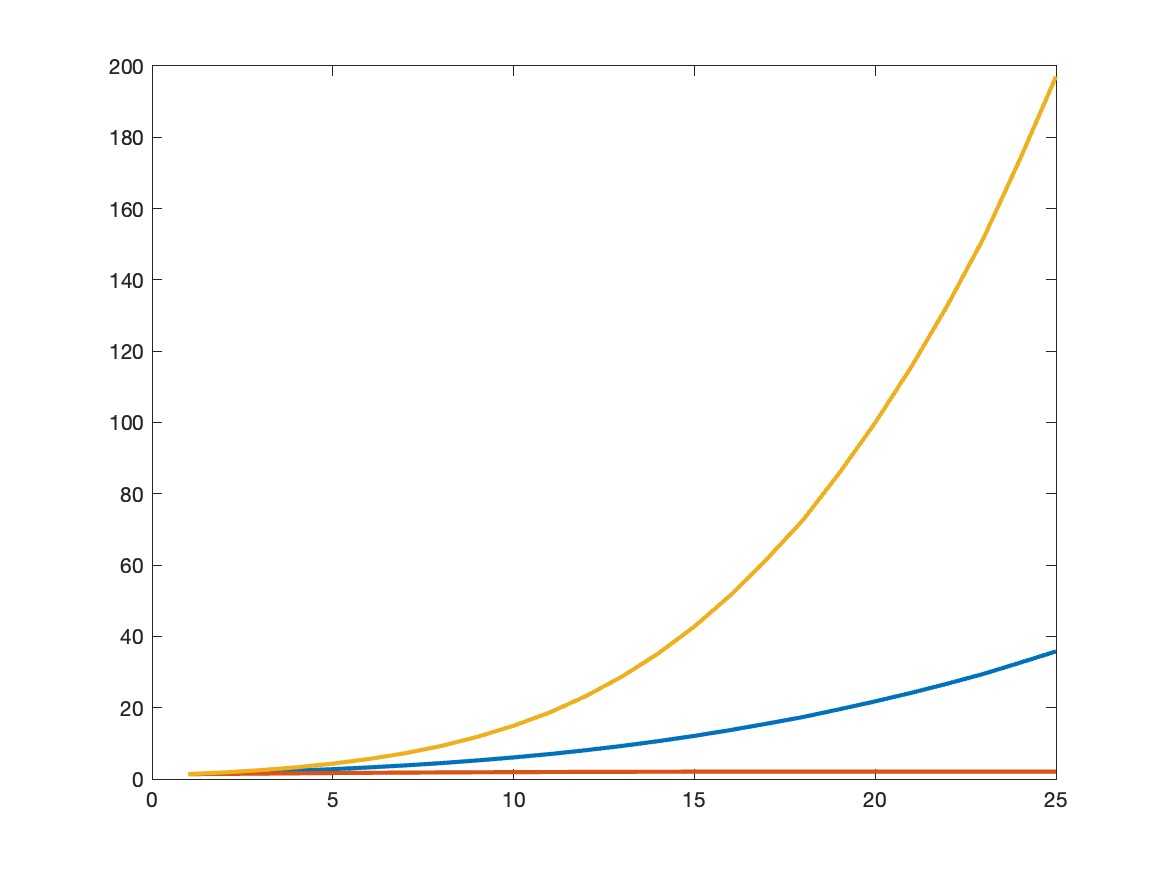}\\
\includegraphics[width=9cm,height=4.5cm]{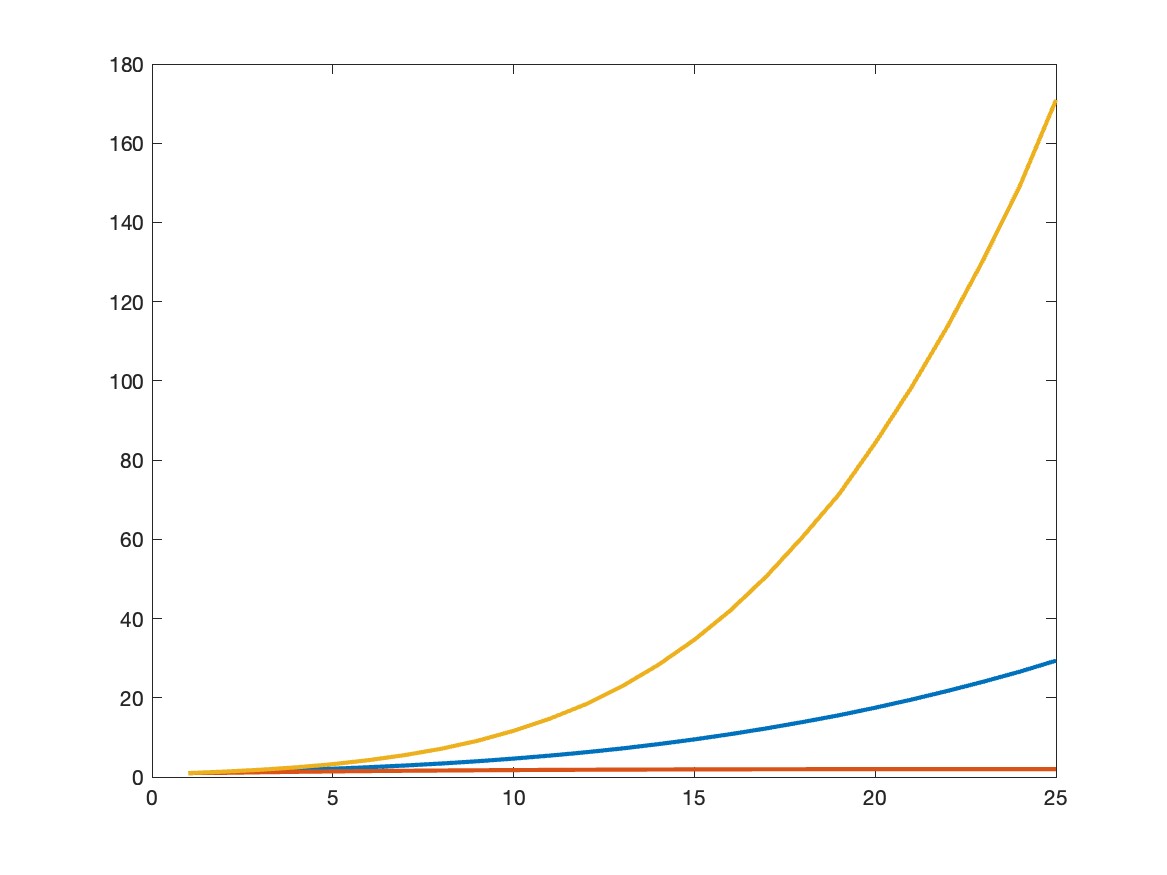}\\
\includegraphics[width=9cm,height=4.5cm]{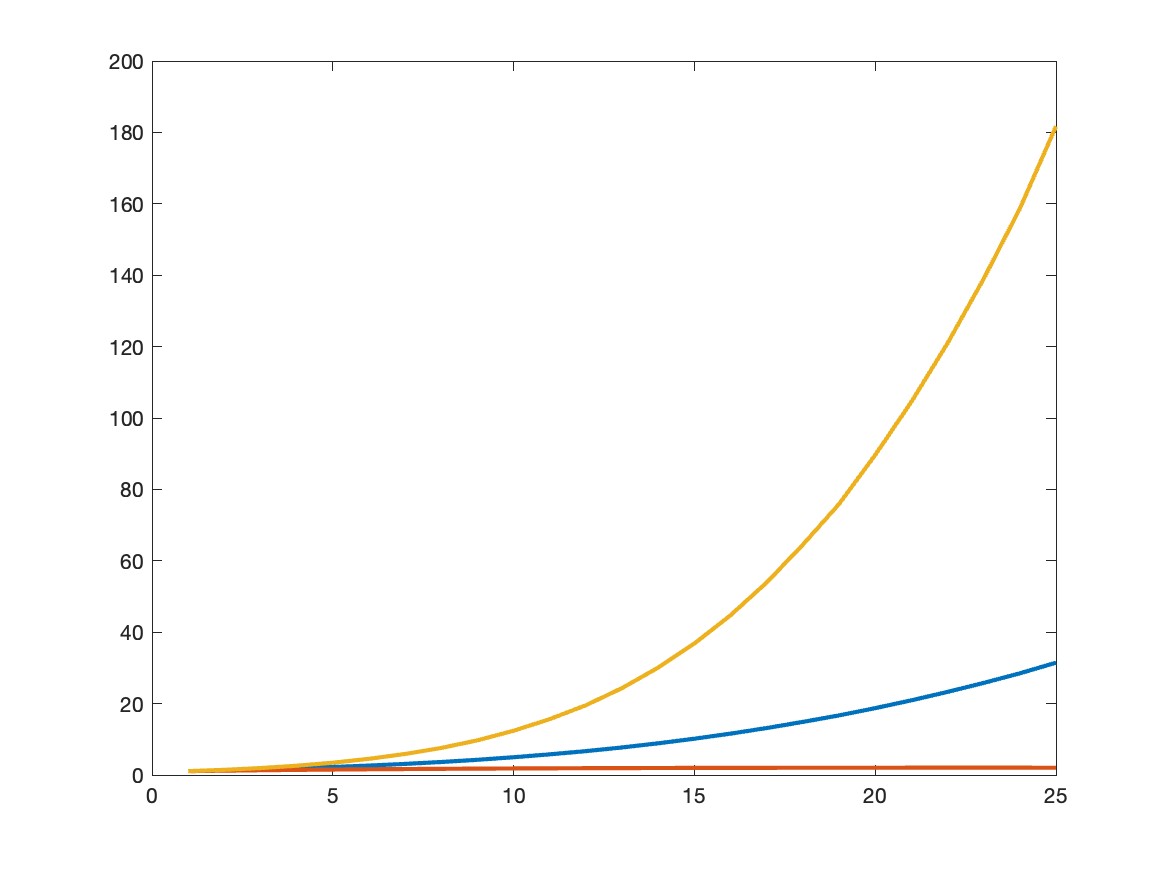}
\end{tabular}
\caption{Nonstationary case.
From top to bottom: $\theta_n/2^{n-1}$,
$\overline{\theta}_n/2^{n-1}$, and
$\widehat{\theta}_n/2¨{n-1}$ versus $n$
(in blue), and the corresponding upper bound
(in orange) and lower bound (in red).
For $\widehat{\theta}_n/2^{n-1}$, $F = \1$.
}\label{fig:nonstat}
\end{figure}

%%%%%%%%%%%%%%%%%%%%%%%%%%%%%%%%%%%%%%%%%%%%%%%%%%%%%%%%%%%%%%%%%%%%%%%%%%%%%%%%%%%%%%%%%%%%%%%%%%%%%%%%%%%%%%%%%%%%%%%%%%%%%%%%%
% CONCLUSION
\section{Conclusion}\label{se:conclusion}
Unrolling optimization algorithms into neural architectures offers significant advantages, including a reduced number of parameters to learn, efficient GPU implementation, and the ability to optimize the quality of the results using relevant loss functions. However, there is currently a lack of theoretical analysis for these structures.\\
In this work, by leveraging tools from fixed-point theory, we conducted an analysis of the robustness of an unrolled version of the forward-backward algorithm, extending the work in \cite{Corbineau2020}, where sensitivity to observed data was overlooked. Our analysis relies on Lipschitz bounds that are computationally efficient (with quadratic complexity) and tighter than the separable bounds commonly used in the neural network community \cite{serrurier2020achieving}.\\
We acknowledge that, \textcolor{black}{although these bounds are rigorous}, they are not optimal. Future work could focus on deriving tighter bounds to further refine the analysis. Additionally, exploring more general proximal algorithms beyond the one considered in this study would be an interesting direction. \textcolor{black}{For example, we could consider the presence of a heavy ball/inertial term in the forward-backward algorithm.}
Finally, we believe the 
$\alpha$-averaging properties we established could inspire the design of recurrent networks in the spirit of \cite{Combettes2019}.

%%%%%%%%%%%%%%%%%%%%%%%%%%%%%%%%%%%%%%%%%%%%%%%%%%%%%%%%%%%%%%%%%%%%%%%%%%%%%%%%%%%%%%%%%%%%%%%%%%%%%%%%%%%%%%%%%%%%%%%%%%%%%%%%%

%\section{Conclusion}\label{sec13}

\backmatter

\bmhead{Acknowledgements}

The work by J.-C. Pesquet was supported by the ANR Chair in Artificial Intelligence BRIDGEABLE. E. Chouzenoux acknowledges support from the European Research Council Starting Grant MAJORIS ERC-2019-STG-850925.


\begin{appendices}

\section{Calculations related to Remark~\ref{re:weightednorm}.}\label{secA0}

\paragraph{Firm nonexpansiveness of $Q_n$.} 
\textcolor{black}{For every 
$n\in \{1,\ldots,m\}$, $z=(x,b)\in \X^2$, and $z'=(x',b')\in \X^2$,
\begin{align*}
&\|Q_n z - Q_n z'\|^2 + \|(\1-Q_n)z+ (\1-Q_n)z'\|^2\\
& = \|R_n x- R_n x'\|^2+\|b-b'\|_{\varpi}^2+ 
\|(\1-R_n)x-(\1-R_n)x'\|^2\\
& \le \|x-x'\|^2+\|b-b'\|_{\varpi}^2\\
& = \|z-z'\|^2\;,
\end{align*}
where the inequality follows from the firm nonexpansiveness of $R_n$.
}
\paragraph{Proof of \eqref{e:defain-bias-weighted}.}
\textcolor{black}{
The introduction of the weighted norm in \eqref{e:normweight} induces
the following changes in the proof of Lemma \ref{e:normUni-bias}.\\
Using the same notation as in the proof, the norm of $U_n\circ \cdots \circ U_i z$ is now given by
\[
\|U_n\circ \cdots \circ U_i z\|^2 = \; \sum_p \left(\beta_{i,n,p}^2 \xp^2 + 2 \beta_{i,n,p} \widetilde{\beta}_{i,n,p} \xp \bp
+ \Big(\widetilde{\beta}_{i,n,p}^2 + \frac{\eta_{i,n}^2}{\varpi_p}\Big) \bp^2\right)
\; .
\]
The above expression has to be maximized on
$(x,b)$ subject to the constraint
\[
\|x\|^2 +\|b\|_\varpi^2=
\sum_p \left(\xp^2 + \frac{\bp^2}{\varpi_p}\right) = 1. 
\]
This suggests to set, for every integer $p$, $\zp' = (\xp, \bp/\sqrt{\varpi}_p)$, which leads to
\[
 \| U_n\circ \cdots \circ U_i \|^2 = \; \underset{z' = (\zp')_p, \|z'\| =1}{\sup}
             \sum_p (\zp')^\top
A'_{i,n,p}\, (\zp') 
\; ,
\]
where the supremum is calculated on the standard
unit ball and 
\[A'_{i,n,p} = \left( 
\begin{array}{cc}
  \beta_{i,n,p}^2   & \sqrt{\varpi_p}\beta_{i,n,p} \widetilde{\beta}_{i,n,p} \\
\sqrt{\varpi_p}\widetilde{\beta}_{i,n,p}\beta_{i,n,p}     & \eta_{i,n}^2 + \varpi_p\widetilde{\beta}_{i,n,p}^2
\end{array}
\right) \; .
\]
%Hence, equipping the product space with the reweighted norm yields
%\begin{equation} \label{e:defnormU2-bias-before}
% \| U_n\circ \cdots \circ U_i \|^2 = \; \underset{z' = (\zp')_p, \|z'\| =1}{\sup}
%             \sum_p (\zp')^\top
%A'_{i,n,p}\, \zp' 
%\; .
%\end{equation}
The rest of the derivation of \eqref{e:defain-bias-weighted} is similar to that in the proof of
Lemma \ref{e:normUni-bias}.
}

\section{Stationary case.}\label{secA1}

\textcolor{black}{This appendix shows that, in the stationary case, some of the quantities of interest can be easily 
expressed. This allows us to state more explicitly necessary/sufficient conditions for the VNN to be  1-Lipschitz.}

In this case, for every 
$p$, $\beta_p^{(n)} = \beta_p$. In addition, $\eta_{1,m} = \eta^m$, 
$\beta_{1,m,p} = \beta_p^m$,
and \eqref{def:vp-bias} simplifies to
\[
\widetilde{\beta}_{1,m,p}
= \frac{\lambda}{1+\lambda\chi} \sum_{j=0}^{m-1}\beta_p^j \eta^{m-j-1}.
\]
Then, it follows from \eqref{e:looserLipb} that  necessary conditions for 
$\theta_m/2^{m-1}=1$
are 
\[
(\forall p)\quad
\begin{cases}
|\beta_p| \le 1\\
\displaystyle\varpi_p \Big(\frac{\lambda}{1+\lambda\chi}\Big)^2\Big(\sum_{j=0}^{m-1}\beta_p^j \eta^{m-j-1}\Big)^2+\eta^{2m}\le 1 \quad \Rightarrow \eta < 1\;.
\end{cases}
\]
The first condition is classical for the convergence of the proximal gradient algorithm and 
it is satisfied if and only if 
\textcolor{black}{$\lambda (\sup_p (\beta_{T,p}+\beta_{D,p})-\chi) \le 2$}.

Furthermore, \eqref{e:looserLipb} yields the following sufficient condition for the VNN to be 1-Lipschitz:
\[
(\forall p)
\quad 
\beta_p^2+\varpi_p \Big(\frac{\lambda}{1+\lambda\chi}\Big)^2+\eta^2 \le 1\;,
\]
which, by virtue of \eqref{def:vp0}, is also equivalent to
\begin{equation}\label{e:sufcondetasstat}
(\forall p)\quad
\big(1-\lambda (\beta_{T,p}+\beta_{D,p})\big)^2+\varpi_p
\lambda^2
+(1+\lambda\chi)^2
(\eta^2-1) \le 0\;.
\end{equation}
As expected, the strong convexity of the regularizing function $g$ allows this condition to be more easily satisfied since $\eta^2-1\le 0$.
\textcolor{black}{
It can be observed that a small value of $\eta$ also has a favorable effect.
}
For simplicity, let us focus on the case when $\chi = 0$, which provides the following sufficient condition for \eqref{e:sufcondetasstat}:
\[
(\forall p)\quad
((\beta_{T,p}+\beta_{D,p})^2+\varpi_p
)
\lambda^2-2 (\beta_{T,p}+\beta_{D,p})\lambda
+\eta^2 \le 0.
\]
These inequalities can be satisfied only if 
\[
(\forall p)\quad
\begin{cases}
(\beta_{T,p}+\beta_{D,p})^2-\eta^2
((\beta_{T,p}+\beta_{D,p})^2+\varpi_p) \ge 0\\
\lambda_{1,p}(\eta) \le 
\lambda \le \lambda_{2,p}(\eta),
\end{cases}
\]
where $\lambda_{1,p}(\eta)$
and $\lambda_{2,p}(\eta)$ are the smallest and largest roots of the
quadratic polynomial
$((\beta_{T,p}+\beta_{D,p})^2+\varpi_p)
\lambda^2-2 (\beta_{T,p}+\beta_{D,p})\lambda
+\eta^2$.
The first inequalities are satisfied if and only if
\[
\eta \le \inf_p \frac{\beta_{T,p}+\beta_{D,p}}{
\sqrt{(\beta_{T,p}+\beta_{D,p})^2+\varpi_p}}
\]
while the second require that
$\sup_p \lambda_{1,p}(\eta) \le  
\inf_p \lambda_{2,p}(\eta)$.




\end{appendices}


\bibliography{biblio}% common bib file

\end{document}